\documentclass[a4paper, reqno, 11pt]{amsart}
\usepackage{amsmath, amssymb, graphics, epsfig, enumerate, amstext, amsthm, cite, verbatim}
\usepackage{xcolor}
\usepackage[plainpages=false,colorlinks,pdfpagemode=None,bookmarksopen,linkcolor=blue,citecolor=blue,urlcolor=blue]{hyperref}
\usepackage{cite}

\newtheorem{theorem}{Theorem}[section]

\newtheorem{corollary}[theorem]{Corollary}
\newtheorem{lemma}[theorem]{Lemma}

\theoremstyle{definition}

\newtheorem{example}[theorem]{Example}

\theoremstyle{remark}
\newtheorem{problem}[theorem]{Problem}

\numberwithin{equation}{section}

\def\IC{{\mathbb C}}
\def\IF{{\mathbb F}}
\def\IR{{\mathbb R}}

\def\bH{{\bf H}}
\def\bM{{\bf M}}

\def\bU{{\bf U}}

\def\b0{{\bf 0}}

\newcommand{\cC}{{\mathcal C}}

\def\tr{{\rm tr}\,}

\def\diag{{\rm diag}\,}

\def\[{\left [}
\def\]{\right ]}
\def\({\left (}
\def\){\right )}

\def\<{{\langle}}
\def\>{{\rangle}}
\def\1{{\bf 1}}
\newcommand{\bprob}{\begin{problem}}
\newcommand{\eprob}{\end{problem}}

\newcounter{example}
\parskip=\baselineskip

\begin{document}
\openup .95\jot
\title[
Nonsurjective maps preserving disjointness, triple products, or norms]{Nonsurjective maps between rectangular matrix spaces
preserving disjointness, triple products, or norms}

\dedicatory{Dedicated to Man-Duen Choi on the occasion of  his 75th birthday}

\author[C.-K. Li, M.-C. Tsai, Y.-S. Wang and N.-C. Wong]{Chi-Kwong Li, Ming-Cheng Tsai, Ya-Shu Wang \and Ngai-Ching Wong}
\address[Li]{Department of Mathematics, The College of William
\& Mary, Williamsburg, VA 13185, USA.}
\email{ckli@math.wm.edu}

\address[Tsai]{General Education Center, Taipei University of Technology 10608, Taiwan.}
\email{mctsai2@mail.ntut.edu.tw}

\address[Wang]{Department of Applied Mathematics, National Chung Hsing University, Taichung 40227, Taiwan.}
\email{yashu@nchu.edu.tw}

\address[Wong]{Department of Applied Mathematics, National Sun Yat-sen
  University, Kaohsiung, 80424, Taiwan.}
  \email{wong@math.nsysu.edu.tw}

\date{\today}


\begin{abstract}
Let $\bM_{m,n}$ be the space of $m\times n$ real or complex rectangular matrices.
Two matrices $A, B \in \bM_{m,n}$ are disjoint if $A^*B = 0_n$ and
$AB^* = 0_m$.
In this paper, a characterization is given for linear maps
$\Phi: \bM_{m,n} \rightarrow \bM_{r,s}$ sending disjoint matrix pairs to
disjoint  matrix pairs, i.e., $A, B \in \bM_{m,n}$ are disjoint
ensures that $\Phi(A), \Phi(B) \in \bM_{r,s}$ are disjoint.
More precisely, it is shown that  $\Phi$
preserves disjointness if and only if $\Phi$ is of the form
$$\Phi(A) =  U\begin{pmatrix}
A \otimes Q_1  & 0 & 0 \cr
0 & A^t \otimes Q_2 & 0 \cr
0 & 0 & 0 \cr\end{pmatrix}V$$
for some unitary
matrices $U \in \bM_{r,r}$ and $V\in \bM_{s,s}$,
and positive diagonal matrices $Q_1, Q_2$, where $Q_1$ or $Q_2$ may be vacuous.
The result is used to characterize nonsurjective linear maps that preserve the $JB^*$-triple product, or just the zero triple product, on rectangular matrices, defined by $\{A,B,C\}
= \frac{1}{2}(AB^*C+CB^*A)$.
The result is also applied to characterize linear maps between
rectangular matrix spaces of different sizes preserving the Schatten $p$-norms or the Ky Fan $k$-norms.
\end{abstract}

\maketitle

\section{Introduction}\label{sec:1}

The fruitful history of linear preserver problems starts with a rather surprising result of Frobenius.  He showed in  \cite{F} that a linear map $\Phi: \bM_n(\IC) \rightarrow \bM_n(\IC)$ of $n\times n$ complex
matrices preserving determinant, i.e., $\det(A)= \det(\Phi(A))$, must be of the form $A \mapsto MAN$ or $A \mapsto MA^tN$ for some matrices $M,N \in \bM_n(\IC)$ with $\det(MN) = 1$.  Another seminal work is due to Kadison.
In  \cite{K51}, Kadison showed that a unital surjective isometry between two $C^*$-algebras ${\mathcal A}$
and ${\mathcal B}$ must be a $C^*$-isomorphism;
in particular, a linear map $\Phi: \bM_n(\IC) \rightarrow \bM_n(\IC)$
leaving the operator norm invariant must be of the form
$A \mapsto UAV$ or $A \mapsto UA^tV$ for some unitary matrices
$U, V \in \bM_n(\IC)$.

Researchers have developed many results and techniques in the
study of linear preserver problems; see, e.g., \cite{Br97, Li01, Li92}.
Many of the results have been extended in different
directions and applied to other topics such as geometrical structure
of Banach spaces, and quantum mechanics;
see, e.g., \cite{FJ02, FJ, Mo07}.
In spite of these advances, there are some intriguing
basic linear preserver problems which remain open.  In particular,
characterizing linear preservers  between different matrix or operator spaces without the surjectivity assumption is very challenging and sometimes intractable; see, for example, \cite{BKS, CL,  LPS, Li02, YCZ, ZC, ZXF}.
Even for finite dimensional spaces, the problem  is highly non-trivial.
For instance, there is no easy description of a linear norm preserver
$\Phi: M_n \rightarrow M_r$ if $n \ne r$; see \cite{CLP04}.

In this paper, we study nonsurjective linear maps between rectangular  matrix spaces
preserving disjointness, the Schatten $p$-norms, or the Ky-Fan $k$-norms. The result is used to characterize linear maps that preserve the $JB^*$-triple product, or just the zero triple product. Note that there are interesting results on disjointness preserving maps on different kinds of products over general operator spaces or algebras, see, e.g., \cite{LTW12,LCLWa,LCLWb,HLW08, HLW10}.
However, the basic problem on disjointness preservers from a
rectangular matrix space to another rectangular matrix space is unknown, and
the existing results do not cover this case. It is our hope that our study will lead to some general techniques for the study of disjointness preservers in a more general context, say, for general $JB^*$-triples, to supplement those established in the few literature, e.g.,  \cite{AP14}.

To better describe the questions addressed in this paper,
we introduce some notation.
Let $\bM_{m,n}$ be the set of $m\times n$ real or complex rectangular matrices,
and let $\bM_n = \bM_{n,n}$.
A pair of matrices $A, B \in \bM_{m,n}$ are \emph{disjoint},
denoted by
$$
A\perp B,\quad\text{if}\quad A^*B=0_n\ \text{and}\ AB^*=0_m.
$$
Here  the adjoint $A^*$ of a rectangular matrix $A$ is  its conjugate transpose $\overline{A^t}$. If $A$ is a real matrix, then $A^*$ reduces to $A^t$, the transpose of $A$. Clearly, $A$ and $B$ are disjoint if and only if they have  orthogonal ranges and initial spaces. A rectangular matrix $A$ is called a \emph{partial isometry}  if $AA^*A=A$. In this case, $A^*A$ is the \emph{range projection} and $AA^*$ is the \emph{initial projection} of $A$. Two partial isometries are {disjoint} if  and only if they have orthogonal range and initial projections.

We will characterize linear
maps $\Phi: \bM_{m,n} \rightarrow \bM_{r,s}$
that preserve disjointness,
i.e., $\Phi(A) \perp \Phi(B)$ whenever $A\perp B$,
and apply the result to some related topics.
In particular, we show in Section \ref{sec:2} that such a map has the form
$$
\Phi(A) =  U\begin{pmatrix}
A \otimes Q_1  & 0 & 0 \cr
0 & A^t \otimes Q_2 & 0 \cr
0 & 0 & 0 \cr\end{pmatrix}V
$$
for some unitary (orthogonal in the real case) matrices
$U \in \bM_r, V\in \bM_s$
and diagonal (square) matrices $Q_1, Q_2$ with positive diagonal entries,
where $Q_1$ or $Q_2$ may be vacuous.

In Section \ref{sec:3}, we regard the space of rectangular matrices as
JB*-triples
carrying the Jordan triple product $\{A,B,C\} = \frac{1}{2}(AB^*C+CB^*A)$,
and use our result in Section \ref{sec:2}
to study JB*-triple homomorphisms on rectangular matrices, i.e.,
linear maps $\Phi: \bM_{m,n}\rightarrow \bM_{r,s}$  satisfying
$$\Phi(AB^*C + CB^*A)= \Phi(A)\Phi(B)^*\Phi(C) +
\Phi(C)\Phi(B)^*\Phi(A)\quad \hbox{ for all } A,B,C \in \bM_{m,n},$$
and also linear maps preserving matrix triples with
zero Jordan triple product.

We also apply our result in Section \ref{sec:2} to study  linear maps
$\Phi: \bM_{m,n} \rightarrow \bM_{r,s}$
preserving the Schatten $p$-norms and the Ky Fan $k$-norms  in Section \ref{sec:4}.
Open problems and future research possibilities are mentioned in Section \ref{sec:5}.

 Throughout the paper, we will always assume that
 $m,n,r,s$ are positive integers, and use the following notation.
\begin{quote}
\begin{enumerate}[{}]\itemindent=-1cm
\item $\bM_{m,n}=\bM_{m,n}(\mathbb{F})$: the vector space of $m\times n$ matrices over $\IF = \IR$ or $\IC$.

\item $\bM_n=\bM_{n}(\mathbb{F})$: the set of $n\times n$ matrices over $\IF = \IR$ or $\IC$.

\item $\bU_n=\bU_n(\mathbb{F}) = \{A \in \bM_n: A^*A = I_n\}$: the set of real orthogonal
or complex unitary matrices depending on $\IF = \IR$ or $\IC$.

\item $\bH_n =\bH_n(\mathbb{F})= \{A \in \bM_n: A = A^*\}$: the set of real symmetric or complex
Hermitian matrices
depending on $\IF = \IR$ or $\IC$.
\end{enumerate}
\end{quote}

\section{Nonsurjective preservers of disjointness}\label{sec:2}

In this section, we will prove the following.

\begin{theorem} \label{thm1}
A linear map $\Phi: \bM_{m,n} \rightarrow \bM_{r,s}$
preserves disjointness, i.e.,
$$
AB^* = 0_m \ \text{and}\ A^*B = 0_n\ \implies\ \Phi(A)\Phi(B)^* = 0_r\ \text{and}\ \Phi(A)^*\Phi(B) = 0_s,
\quad\forall A, B \in \bM_{m,n},
$$
if and only if
there exist $U \in \bU_r, V\in \bU_s$ and
diagonal matrices $Q_1, Q_2$
with positive diagonal entries such that
\begin{equation}\label{s-form}
\Phi(A) =  U\begin{pmatrix}
A \otimes Q_1  & 0 & 0 \cr
0 & A^t \otimes Q_2 & 0 \cr
0 & 0 & 0 \cr\end{pmatrix}V \qquad \hbox{ for all } A \in \bM_{m,n}.
\end{equation}
Here $Q_1$ or $Q_2$,
may be vacuous.
\end{theorem}

Several  remarks are in order concerning Theorem \ref{thm1}.
\begin{enumerate}
\item Observing the symmetry and avoiding the triviality,  we can  assume that
$2 \le m \le n$.

\item $AB^* = 0_m$ and $A^* B = 0_n$ mean that $A$ and $B$ have orthogonal ranges and orthogonal initial spaces.  This amounts to saying that
we can obtain their singular value decompositions,  $UAV = \sum_{j=1}^{k} a_j E_{jj}$ and
$UBV = \sum_{j=k+1}^p b_j E_{jj}$, for some positive scalars $a_1, .., a_k$, $b_{k+1}, ..., b_p$, and unitary matrices
   $U\in \bU_m$ and $V\in \bU_n$.

\item In view of the singular value decompositions, (\ref{s-form}) in Theorem \ref{thm1} holds if the condition
$$
\Phi(E)\,\bot\, \Phi(F)\quad\text{whenever}\quad E\,\bot\, F
$$
is verified
just for rank one  disjoint partial isometries $E,F$ in $\bM_{m,n}$.

\item
In Theorem \ref{thm1}, unless $r\geq m$ and $s\geq n$, or $s\geq m$ and $r\geq n$,
$\Phi$ will  be the zero map. If $(m,n) = (r,s)$ (resp.\ $(s,r)$) and $m \ne n$,
then $\Phi$ will be
the zero map or  of the form $A \mapsto UAV$ (resp.\ $A \mapsto UA^t V$)
with $U \in \bU_r, V \in \bU_s$.

\item By relaxing the terminology, the rectangular matrix $A \otimes Q_1$
is permutationally similar to
$q_{1} A \oplus \cdots \oplus q_{r} A$
if $Q_1 = \diag(q_1, \dots, q_r)$.
Similarly $A^t\otimes Q_2$ is permutationally similar to
a direct sum of positive multiples of $A^t$.
So, the theorem asserts that up to a fixed unitary equivalence
$\Phi(A)$ is a direct sum of positive multiples of $A$ and $A^t$.

\item
In addition to real and complex rectangular matrices, the conclusions in Theorem \ref{thm1}
is also valid with the same proof for a real linear map
$\Phi: \bH_n \rightarrow \bM_{r,s}$ preserving disjointness. We can further assume that
the co-domain is $\bH_r$, i.e., $\Phi: \bH_n \rightarrow \bH_r$. In this case,  the disjointness assumption on $\Phi$ reduces to that  $AB = 0$ implies $\Phi(A)\Phi(B) = 0$. Adjusting the proof of Theorem \ref{thm1}, we can achieve the equality $U = V^*$, at the expenses that the diagonal matrices $Q_1, Q_2$
may have negative entries.

\item
If the domain  is the set $\bM_n(\mathbb{C})$ of $n\times n$ complex matrices or
the set $\bH_n(\mathbb{C})$ of $n\times n$ complex Hermitian matrices,
our results can be deduced from the abstract theorems on
$C^*$-algebras; e.g., see \cite{LCLWb,LTW12,LW13,BFGMP08}, and also \cite{LCLWa,CKLW03}.
However, the proofs there do not seem to work for
rectangular matrix spaces, or real square matrix spaces.

\item
Our proof is computational and long. It would be nice to have some short and conceptual proofs.
\end{enumerate}

The rest of the section is devoted to the proof of Theorem \ref{thm1}.
We describe our proof strategy.
Let $\{E_{11}, E_{12}, \dots, E_{mn}\}$ be the
standard basis for $\bM_{m,n}$.
We will show that one can apply a series of replacements of $\Phi$ by mappings of the form $X \mapsto \tilde U \Phi(X) \tilde V$ for some  $\tilde U \in \bU_r, \tilde V \in \bU_s$ so that the resulting map satisfies
\begin{equation*}
E_{ij}\mapsto \begin{pmatrix}
E_{ij} \otimes Q_1 & 0 & 0 \cr
0 &  E_{ji} \otimes Q_2 & 0 \cr 0 & 0 & 0 \cr\end{pmatrix}
\qquad \hbox{ for all } 1 \le i \le m,\ 1 \le j \le n.
\end{equation*}
The result will then follow. We carry out the above scheme with an inductive argument, and divide the proofs into several lemmas.

Note that in this section only the linearity and the disjointness structure of the rectangular matrices are concerned. As will be shown below, the (real or complex) matrix space $\bM_2 = \operatorname{span}\{E_{11}, E_{12}, E_{21}, E_{22}\}$
and  the matrix space $\operatorname{span}\{E_{ij}, E_{ik}, E_{lj}, E_{lk}\}$
can be considered as the  same object during our discussion.

\begin{lemma}\label{lem:M2all}
Let $i\neq l$ and $j\neq k$.
The bijective linear map $\Psi: \bM_2\to \operatorname{span}\{E_{ij}, E_{ik}, E_{lj}, E_{lk}\}$,
sending $E_{11}, E_{12}, E_{21}, E_{22}$ to $E_{ij}, E_{ik}, E_{lj}, E_{lk}\in\bM_{m,n}$ respectively, preserves the disjointness in two directions, i.e.,
$$
A\,\bot\, B \quad\Longleftrightarrow\quad \Psi(A)\,\bot\, \Psi(B)\qquad \hbox{ for all }  A, B\in \bM_2.
$$
\end{lemma}
\begin{proof}
The assertion follows from the fact that $\Psi(A)=UAV$, where $U=E_{i1}+E_{l2}\in\bM_{m,2}$ and $V=E_{1j}+E_{2k}\in\bM_{2,n}$ are partial isometries
such that $U^*U=VV^*=I_2$, the $2\times 2$ identity matrix.
\end{proof}

The technical lemma  will be
used heavily in the subsequent proofs.  Although the statement is stated and proved
for the case when the domain is $\bM_2$, it is indeed valid for all the rectangular matrix space
$\operatorname{span}\{E_{ij}, E_{ik}, E_{lj}, E_{lk}\}$ due to Lemma \ref{lem:M2all}.
In the future application,
the lemma ensures that if $\Phi(E_{ij})$ and
$\Phi(E_{lk})$ have some nice structure for a disjointness preserving linear map
$\Phi: \bM_{m,n} \rightarrow \bM_{r,s}$,
then  much can be said about $\Phi(E_{ik}+E_{lj})$ and
$\Phi(E_{ik}-E_{lj})$. One can then compose $\Phi$ with some
unitaries so that all $\Phi(E_{ij}), \Phi(E_{ik}), \Phi(E_{lj})$ and $\Phi(E_{lk})$ have simple structure.

\begin{lemma} \label{lem1}
Let $\Phi: \bM_{2} \rightarrow \bM_{r,s}$ be a nonzero linear map
preserving disjointness such that
$$\Phi(E_{11}) = \begin{pmatrix} D_1  & 0 & 0  \cr 0 & 0_\ell  & 0 \cr
0 & 0 & 0_{r-k-\ell, s-k-\ell} \end{pmatrix} \quad \hbox{ and } \quad
\Phi(E_{22}) =
\begin{pmatrix} 0_k  & 0 & 0  \cr 0 & D_2  & 0 \cr
0 & 0 & 0_{r-k-\ell, s-k-\ell} \end{pmatrix},$$
where $D_1 \in \bM_k, D_2 \in \bM_\ell$ are diagonal matrices with positive
diagonal entries arranged in descending order, and
$D_1  = \alpha_1 I_{u_1} \oplus \cdots \oplus \alpha_v I_{u_v}$
with $\alpha_1 > \cdots > \alpha_v > 0$ and $u_1 + \cdots + u_v = k$.
\begin{enumerate}[(a)]
\item We have $D_1 = D_2$.
Moreover,
$$
\Phi(E_{12}+E_{21}) = \begin{pmatrix} 0_k  & B_{12} & 0  \cr
B_{12}^* & 0_k  & 0 \cr
0 & 0 & 0_{r-2k,s-2k} \end{pmatrix} \ \hbox{ and } \
\Phi(E_{12}-E_{21}) = \begin{pmatrix} 0_k  & {\widehat C_{12}} & 0  \cr
-{\widehat C_{12}}^* & 0_k  & 0 \cr
0 & 0 & 0_{r-2k,s-2k} \end{pmatrix},
$$
where $B_{12} = \alpha_1 W_1 \oplus \cdots \oplus \alpha_v W_v$
and
${\widehat C_{12}} = \alpha_1 W_1 V_1 \oplus \cdots
\oplus \alpha_v W_v V_v$ with $W_j, V_j \in \bU_{u_j}$.

\item There are unitaries $R_1, R_2\in \bU_k$ and a permutation $P\in \bM_k$ such that the map
$$X \mapsto (P^*R_2^*R_1^* \oplus P^*R_2^*
\oplus I_{r-2k})\Phi(X)(R_1R_2P \oplus R_2P  \oplus I_{s-2k})$$
satisfies
$$
E_{11} \mapsto  \begin{pmatrix}
Q_1\oplus Q_2 & 0_k & 0 \cr
0_k & 0_k & 0 \cr
0 & 0 & 0\cr\end{pmatrix},
 \quad
E_{12} \mapsto \begin{pmatrix}
0_k & Q_1 \oplus 0_{k_2} & 0 \cr
0_{k_1}\oplus Q_2  & 0_k & 0 \cr
0 & 0 & 0\cr\end{pmatrix},
$$
$$
E_{21} \mapsto \begin{pmatrix}
0_k & 0_{k_1} \oplus Q_2 & 0 \cr
Q_1 \oplus 0_{k_2} & 0_k & 0 \cr
0 & 0 & 0\cr\end{pmatrix}, \quad
E_{22} \mapsto  \begin{pmatrix}
0_k & 0_k  & 0 \cr
0_k & Q_1 \oplus Q_2 & 0 \cr
0 & 0 & 0\cr\end{pmatrix},
$$
where $Q_1\in \bM_{k_1}$, $Q_2\in \bM_{k_2}$, $k_1+k_2 = k$,
are diagonal matrices with positive diagonal entries from $\{\alpha_1, \ldots, \alpha_v\}$  arranged in
descending order.
\end{enumerate}
\end{lemma}
\begin{proof}
(a)\ Suppose $\Phi: \bM_{2} \rightarrow \bM_{r,s}$ satisfies the assumption. Let
$$\Phi(E_{12}+E_{21})
=\begin{pmatrix} B_{11} & B_{12} & B_{13} \cr B_{21} & B_{22} & B_{23} \cr
B_{31} &  B_{32} & B_{33} \cr\end{pmatrix},$$
where $B_{11} \in \bM_k, B_{22} \in \bM_\ell$.
For every nonzero $\gamma\in \mathbb{R}$,
the pair of the matrices
$$Z_1
= \begin{pmatrix} \gamma & 1 \cr 1 & \frac{1}{\gamma} \cr\end{pmatrix} \quad
\hbox{ and } \quad
Z_2 = \begin{pmatrix} \frac{1}{\gamma} & -1 \cr -1 & \gamma \cr\end{pmatrix}$$
are  disjoint, and so are the pair $T_1 = \Phi(Z_1)$ and $T_2 = \Phi(Z_2)$.
Considering the $(1,1)$, $(1,2)$, $(2,1)$,
$(2,2)$, $(3,3)$ blocks of the matrix $T_1^*T_2$,
we get the following:
\begin{eqnarray*}
0_k &=&  D_1^2+\frac{1}{\gamma}B_{11}^*D_1-\gamma D_1B_{11}-B_{11}^*B_{11}-
B_{21}^*B_{21} - B_{31}^* B_{31}, \cr
0_{k,\ell} &=& \gamma(B_{21}^*D_2-D_1B_{12})-B_{11}^*B_{12}-B_{21}^*B_{22} -
B_{31}^* B_{32}, \cr
0_{\ell,k} &=&  \frac{1}{\gamma}(B_{12}^*D_1-D_2B_{21})-B_{12}^*B_{11}-B_{22}^*B_{21}
- B_{32}^*B_{31}, \cr
0_\ell  &=& D_2^2-\frac{1}{\gamma}D_2B_{22}+\gamma B_{22}^*D_2-B_{22}^*B_{22}-
B_{12}^*B_{12} - B_{32}^*B_{32}, \cr
0_{s-k-\ell} & = & -B_{13}^*B_{13}-B_{23}^* B_{23} - B_{33}^*B_{33}.
\end{eqnarray*}
Considering the  $(1,1)$, $(1,2)$, $(2,1)$,
$(2,2)$, $(3,3)$ blocks of the matrix $T_1T_2^*$,
we get the following:
\begin{eqnarray*}
0_k &=&  D_1^2+\frac{1}{\gamma}B_{11}D_1-\gamma D_1B_{11}^*-B_{11}B_{11}^*
  -B_{12}B_{12}^* - B_{13}B_{13}^*, \cr
0_{k,\ell} &=& \gamma(B_{12}D_2-D_1B_{21}^*)-B_{11}B_{21}^*-B_{12}B_{22}^*-
    B_{13}B_{23}^*, \cr
0_{\ell,k} &=&\frac{1}{\gamma}(B_{21}D_1-D_2B_{12}^*)-B_{21}B_{11}^*-B_{22}B_{12}^*
- B_{23}B_{13}^*, \cr
0_\ell &=&
  D_2^2-\frac{1}{\gamma}D_2B_{22}^*+\gamma B_{22}D_2-B_{21}B_{21}^*-B_{22}B_{22}^*
- B_{23}B_{23}^*, \cr
0_{r-k-\ell} &=& -B_{31}B_{31}^* - B_{32}B_{32}^* - B_{33}B_{33}^*.
\end{eqnarray*}
In view of the $(3,3)$ blocks of $T_1^*T_2$ and $T_1T_2^*$ being zero blocks,
we see that $B_{13}, B_{23}, B_{33}, B_{31}, B_{32}$ are zero blocks.
Since $0\neq\gamma$ is arbitrary and $D_1$, $D_2$ are invertible, we see
that
\begin{gather}
B_{11}= 0_k,\quad  B_{22}=0_\ell, \notag \\
B_{12}B_{12}^*=B_{21}^*B_{21}=D_1^2 \in \bM_k, \quad
B_{12}^*B_{12}=B_{21}B_{21}^*=D_2^2 \in \bM_\ell,\label{eq1} \\
D_1B_{12}=B_{21}^*D_2, \quad \hbox{ and } \quad B_{12}D_2=D_1B_{21}^*. \label{eq2}
\end{gather}
Note that $B_{12}B_{12}^*$ and $B_{12}^*B_{12}$ have the same nonzero eigenvalues
(counting multiplicities). Because  $D_1, D_2$ have positive
diagonal entries arranged in descending order, it follows from (\ref{eq1}) that
$k =  \ell$ and  $D_1=D_2$.

We can now assume that
$D_1  =  D_2 = \alpha_1 I_{u_1} \oplus \cdots \oplus \alpha_v I_{u_v}$
with $\alpha_1 > \cdots > \alpha_v > 0$ and $u_1 + \cdots + u_v = k$.
Furthermore, from (\ref{eq1})
the matrices $B_{12}, B_{12}^*, B_{21}$ and $B_{21}^*$ have orthogonal
columns with Euclidean norms equal to the diagonal entries of $D_1$.
By (\ref{eq2}),  we see that
$$
B_{12} = B_{21}^* = \alpha_1 W_1 \oplus \cdots \oplus \alpha_v W_v
$$
for some $W_1 \in \bU_{u_1}, \dots, W_v \in \bU_{u_v}$.

Let $R_1 = W_1 \oplus \cdots \oplus W_v$.
Replace
$\Phi$ by
$X \mapsto (R_1^* \oplus I_{r-k})\Phi(X)(R_1 \oplus I_{s-k})$. We may assume that $B_{12} = B_{21}^* = D_1$.
Let  $$\Phi(E_{12}-E_{21})
=\begin{pmatrix} C_{11} & C_{12} & C_{13} \cr C_{21} & C_{22} & C_{23} \cr
C_{31} & C_{32} & C_{33} \cr\end{pmatrix},$$
where $C_{11} \in \bM_k, C_{22} \in \bM_\ell$.

Now, the pair of matrices
$$
Z_3 = \begin{pmatrix} \gamma & -1 \cr 1 & -\frac{1}{\gamma} \cr\end{pmatrix}
\quad \text{ and } \quad
Z_4 = \begin{pmatrix} \frac{1}{\gamma} & 1 \cr -1 & -\gamma \cr\end{pmatrix}
$$
are disjoint, and so are the pair of matrices
$T_3=\Phi(Z_3)$ and $T_4 = \Phi(Z_4)$.
Consider the $(1,1)$, $(1,2)$, $(2,1)$,
$(2,2)$, $(3,3)$ blocks of the matrix $T_3^*T_4$.
By the fact that $k = \ell$ and $D_1 = D_2$, we get the following:
\begin{eqnarray*}
0_k &=&  D_1^2-\frac{1}{\gamma}C_{11}^*D_1+\gamma D_1C_{11}-C_{11}^*C_{11}
-C_{21}^*C_{21} - C_{31}^*C_{31}, \cr
0_k &=&
\gamma(D_1C_{12}+C_{21}^*D_2)-C_{11}^*C_{12}-C_{21}^*C_{22} - C_{31}^*C_{32}, \cr
0_{k} &=&
-\frac{1}{\gamma}(C_{12}^*D_1+D_2C_{21})-C_{12}^*C_{11}-C_{22}^*C_{21}
- C_{32}^*C_{31}, \cr
0_k &=& D_2^2-\frac{1}{\gamma}D_2C_{22}+\gamma C_{22}^*D_2-C_{22}^*C_{22}-
C_{12}^*C_{12} - C_{32}^*C_{32},\cr
0_{s-2k}&=& -C_{13}^*C_{13}- C_{23}^*C_{23} - C_{33}^*C_{33}.
\end{eqnarray*}
Consider the $(1,1)$, $(1,2)$, $(2,1)$,
$(2,2)$, $(3,3)$ blocks of the matrix $T_3T_4^*$.
We get the following:
\begin{eqnarray*}
0_k &=&  D_1^2-\frac{1}{\gamma}C_{11}D_1+\gamma D_1C_{11}^*-C_{11}C_{11}^*
-C_{12}C_{12}^* - C_{13}C_{13}^*,\cr
0_k &=& \gamma(D_1C_{21}^*+C_{12}D_2)-C_{11}C_{21}^*-C_{12}C_{22}^*
- C_{13} C_{23}^*,\cr
0_k &=&
 -\frac{1}{\gamma}(C_{21}D_1+D_2C_{12}^*)-C_{21}C_{11}^*-C_{22}C_{12}^*
 - C_{23}C_{13}^*,\cr
 0_k  &=& D_2^2-\frac{1}{\gamma}D_2C_{22}^*
 +\gamma C_{22}D_2-C_{21}C_{21}^*-C_{22}C_{22}^* - C_{23} C_{23}^*, \cr
 0_{r-2k} &=& -C_{31}C_{31}^* - C_{32} C_{32}^* - C_{33} C_{33}^*.
 \end{eqnarray*}
By a similar argument  for the pair $(T_1, T_2)$, we conclude that $C_{11}$, $C_{22}$, $C_{13}, C_{23}, C_{33}, C_{31}, C_{32}$ are zero blocks. Furthermore,

\medskip\centerline{
$C_{21}^*C_{21}=C_{12}C_{12}^* = C_{21}C_{21}^*=C_{12}^*C_{12}=D_1^2$, \
$D_1C_{12}=-C_{21}^*D_1$, \ and \ $C_{12}D_1=-D_1C_{21}^*$. }

\medskip\noindent
Now, $C_{21}, C_{12}^*, C_{21}^*, C_{12}$ have orthogonal columns with Euclidean norms  equal to the diagonal entries of $D_1$, and together with the fact that
$D_1C_{12}=-C_{21}^*D_1$,  and  $C_{12}D_1=-D_1C_{21}^*$, we see that
$$C_{12} = - C_{21}^* = \alpha_1 V_1 \oplus \cdots  \oplus \alpha_vV_v
\in \bM_{u_1} \oplus \cdots \oplus \bM_{u_v},$$
where $V = D_1^{-1}C_{12} = V_1 \oplus \cdots \oplus V_v$ is unitary.
Thus in its original form, we see that
$$
{\widehat C_{12}} = -{\widehat C_{21}^*}  = \alpha_1 W_1 V_1 \oplus \cdots
\oplus \alpha_v W_v V_v.
$$

(b)
Continue the arguments in (a), and in particular assume that $B_{12}=B_{21}^*=D_1$ and $C_{12} = - C_{21}^* = \alpha_1 V_1 \oplus \cdots  \oplus \alpha_vV_v= D_1V$.
There is a unitary matrix
$R_2 = U_1 \oplus \cdots \oplus U_v \in \bU_k$ with
$U_1 \in \bM_{u_1}, \dots, U_v \in \bM_{u_v}$ such that
$R_2^*VR_2 = \diag(g_1, \dots, g_k) = G \in \bU_k$.
Now, we may replace $\Phi$ by the map
$X \mapsto (R_2^*\oplus R_2^* \oplus I_{r-2k})\Phi(X) (R_2\oplus R_2 \oplus I_{s-2k})$
and assume that $C_{12} = -C_{21}^* = D_1G$.
In particular,
\begin{align*}
\Phi(E_{12}+E_{21}) = \begin{pmatrix} 0_k  & D_{1} & 0  \cr
D_1 & 0_k  & 0 \cr
0 & 0 & 0_{r-2k,s-2k} \end{pmatrix} \ \hbox{ and } \
\Phi(E_{12}-E_{21}) = \begin{pmatrix} 0_k  & {D_1 G} & 0  \cr
- D_1 G^* & 0_k  & 0 \cr
0 & 0 & 0_{r-2k,s-2k} \end{pmatrix}.
\end{align*}

We claim that $G$ is permutationally similar to
$I_{k_1} \oplus -I_{k_2}$ with $k_1 + k_2 = k$.
To see this, consider the pair
$$\Phi(E_{12})=\begin{pmatrix} 0_k & \frac{D_1(I_k+G)}{2} & 0 \cr
\frac{D_1(I_k-G^*)}{2} & 0_k & 0 \cr
0 & 0 & 0 \end{pmatrix},  \quad
\Phi(E_{21})=\begin{pmatrix} 0_k & \frac{D_1(I_k-G)}{2} & 0 \cr
\frac{D_1(I_k+G^*)}{2} & 0_k & 0 \cr
0 & 0 & 0 \cr \end{pmatrix}.$$
One readily checks that the pair are disjoint if and only if
$(I_k+G)(I_k-G^*) = 0_k$, equivalently,
$G$ is a real diagonal unitary matrix.
Thus, there is a permutation matrix $P \in \bM_k$
such that  $P^tGP
= I_{k_1} \oplus -I_{k_2}$ with $k_1 + k_2 = k$.
With a further permutation, we can assume $P^tD_1GP=Q_1 \oplus -Q_2 \in \bM_k$ so that $Q_1, Q_2$ are  diagonal
matrices with descending positive diagonal entries.

We may replace $\Phi$ by a map
$$X \mapsto (P^t\oplus P^t \oplus I_{r-2k})\Phi(X)(P\oplus P \oplus I_{s-2k})$$
so that
$$
\Phi(E_{11}) = \begin{pmatrix}
Q_1\oplus Q_2 & 0_k & 0 \cr
0_k & 0_k & 0 \cr
0 & 0 & 0\cr\end{pmatrix}, \quad
\Phi(E_{12}+E_{21}) = \begin{pmatrix}
0_k & Q_1 \oplus Q_2 & 0 \cr
Q_1 \oplus Q_2 & 0_k & 0 \cr
0 & 0 & 0\cr\end{pmatrix},$$
$$\Phi(E_{22}) = \begin{pmatrix}
0_k & 0_k  & 0 \cr
0_k & Q_1 \oplus Q_2 & 0 \cr
0 & 0 & 0\cr\end{pmatrix}, \quad
\Phi(E_{12}-E_{21}) = \begin{pmatrix}
0_k & Q_{1} \oplus -Q_2 & 0 \cr
-Q_1 \oplus Q_2 & 0_k & 0 \cr
0 & 0 & 0\cr\end{pmatrix}.
$$
Adding and subtracting the matrices
$\Phi(E_{12}+E_{21})$ and $\Phi(E_{12}-E_{21})$,
we get the desired forms of $\Phi(E_{12})$ and $\Phi(E_{21})$.
The result follows.
\end{proof}

\begin{lemma} \label{lem2}
Theorem \ref{thm1} holds if $m = n \ge 2$.
\end{lemma}

\begin{proof}\rm
We prove the result by induction on $m = n \ge 2$.
Suppose $m = n = 2$.
We may choose $V_1 \in \bU_r, V_2 \in \bU_s$
such that
$$Y_1 = V_1\Phi(E_{11})V_2 = \begin{pmatrix} D_1  & 0 & 0  \cr 0 & 0_\ell  & 0 \cr
0 & 0 & 0 \end{pmatrix} \quad \hbox{ and } \quad
Y_2 = V_1\Phi(E_{22})V_2 =
\begin{pmatrix} 0_k  & 0 & 0  \cr 0 & D_2  & 0 \cr
0 & 0 & 0 \end{pmatrix},$$
where $D_1 \in \bM_k, D_2 \in \bM_\ell$ are diagonal matrices with positive
diagonal entries arranged in descending order.
We may replace $\Phi$ by the map
$X \mapsto V_1\Phi(X)V_2$ so that the resulting map
will preserve disjointness and send $E_{jj}$ to $Y_j$ for $j = 1,2$.
By Lemma \ref{lem1}, we can modify $V_1$ and $V_2$ so that the resulting
map satisfies
\begin{align*}
\Phi(E_{11}) &= \begin{pmatrix}
Q_1\oplus Q_2 & 0_k & 0 \cr
0_k & 0_k & 0 \cr
0 & 0 & 0\cr\end{pmatrix}, \quad
\Phi(E_{12}) =\begin{pmatrix}
0_k & Q_1 \oplus 0_{k_2} & 0 \cr
0_{k_1}\oplus Q_2  & 0_k & 0 \cr
0 & 0 & 0\cr\end{pmatrix},
\\
\Phi(E_{21}) &= \begin{pmatrix}
0_k & 0_{k_1} \oplus Q_2 & 0 \cr
Q_1 \oplus 0_{k_2} & 0_k & 0 \cr
0 & 0 & 0\cr\end{pmatrix}, \quad
\Phi(E_{22}) = \begin{pmatrix}
0_k & 0_k  & 0 \cr
0_k & Q_1 \oplus Q_2 & 0 \cr
0 & 0 & 0\cr\end{pmatrix},
\end{align*}
for some diagonal matrices $Q_1,Q_2$ with descending positive diagonal entries.

Now, we can find a permutation matrix $\hat P \in \bM_{2k}$ satisfying
$[X_1| X_2 | X_3|X_4]\hat P = [X_1|X_3 |X_2| X_4]$
whenever  $X_1, X_3 \in \bM_{2k,k_1},
X_2, X_4 \in \bM_{2k,k_2}$. Then
the map
$X \mapsto (\hat P \oplus I_{r-2k})^t\Phi(X)(\hat P \oplus I_{s-2k})$
will satisfy
$$E_{ij} \mapsto  \begin{pmatrix}
E_{ij}\otimes Q_1  & 0_{2k_1,2k_2} & 0_{2k_1,s-2k} \cr
0_{2k_2,2k_1} & E_{ji} \otimes Q_2  & 0_{2k_2,s-2k} \cr
0_{r-2k,2k_1} & 0_{r-2k,2k_2} & 0_{r-2k,s-2k}\cr\end{pmatrix}
\qquad \hbox{ for } 1 \le i, j \le 2.$$
This establishes the assertion for the case when $m=n=2$.

Now, suppose the result holds for square matrices of size smaller than
$n$ with $n > 2$.
Then the restriction of $\Phi$ on matrices $A \in \bM_n$ with the last row and last column equal to zero verifies the conclusion. So, there exist  $U \in \bU_r$ and $V \in \bU_s$ such that
$$U\Phi(E_{ij})V = \begin{pmatrix}
\hat E_{ij} \otimes Q_1 & 0_{(n-1)k_1,(n-1)k_2} & 0\cr
0_{(n-1)k_2,(n-1)k_1} & \hat E_{ji}\otimes Q_2 & 0  \cr
0 & 0 & 0_{r-(n-1)k, s-(n-1)k} \cr\end{pmatrix} \qquad \hbox{ for } 1 \le i, j < n.$$
Here $\{E_{ij}:  1 \le i,  j \le n\}$
is the standard basis for $\bM_{n}$, and
$\{\hat E_{ij}:  1 \le i, j \le n-1\}$
is the standard basis for $\bM_{n-1}$,
$Q_1 \in \bM_{k_1}, Q_2 \in \bM_{k_2}$ are diagonal matrices with
positive diagonal entries, and $k = k_1+k_2$.

Note that $E_{nn}$ and $E_{ij}$ are disjoint for all $1 \le i, j < n$.
So, we may assume that
$$\Phi(E_{nn}) = \begin{pmatrix} 0_{(n-1)k} & 0 \cr 0 & Y\cr\end{pmatrix}$$
for some matrix $Y\in \bM_{r-(n-1)k,s-(n-1)k}$. There exist
$U_1\in \bU_{r-(n-1)k}, V_1\in \bU_{s-(n-1)k}$ such that
$$U_1YV_1 = \begin{pmatrix} D & 0 \cr 0 & 0 \cr\end{pmatrix},$$
where $D$ is a diagonal matrix with positive
diagonal entries arranged in descending order.
We may replace $\Phi$ by the map
$$X \mapsto (I_{(n-1)k} \oplus U_1)\Phi(X)(I_{(n-1)k} \oplus V_1)$$
and assume that $U_1 = I_{r-(n-1)k}$ and $V_1 = I_{s-(n-1)k}.$

Consider the restriction of the
map on the $\operatorname{span}\{E_{11}, E_{1n}, E_{n1}, E_{nn}\}$.
Applying the proof of Lemma \ref{lem1} to the restriction map,
we see that there is a permutation matrix $P$ such that
$D = P^t(Q_1 \oplus Q_2)P$.
Now, replace $\Phi$ by the map
$$
X \mapsto ((I_{n-1}\otimes P^t) \oplus I_{r-(n-1)k})
\Phi(X)((I_{n-1} \otimes P) \oplus I_{s-(n-1)k}).
$$
After a further permutation, we can replace $\hat E_{ij}$ with $E_{ij}$ for $1 \le i, j < n$, and
 the resulting map $\Phi$ satisfies
\begin{equation}\label{ejj}
E_{jj} \mapsto \begin{pmatrix} E_{jj} \otimes D & 0 \cr 0 & 0 \cr\end{pmatrix},
\quad j = 1, \dots, n,
\end{equation}
$$E_{ij}+E_{ji}\mapsto
\begin{pmatrix} (E_{ij} + E_{ji}) \otimes D & 0 \cr
0 & 0_{r-(n-1)k,s-(n-1)k}\cr\end{pmatrix}, \quad 1 \le i \le j < n,$$
$$E_{ij}-E_{ji} \mapsto \begin{pmatrix}
(E_{ij} - E_{ji}) \otimes \hat D & 0 \cr
0 & 0_{r-(n-1)k,s-(n-1)k} \cr\end{pmatrix}, \quad 1 \le i < j < n,$$
where $\hat D = P^t(Q_1 \oplus -Q_2)P$.

For $j = 1, 2, \dots, n-1$, apply  Lemma \ref{lem1}(a)
to the restriction map on the rectangular matrix space
$\operatorname{span} \{E_{jj}, E_{jn}, E_{nj}, E_{nn}\}$.
We see that
$$ \Phi(E_{jn}+E_{nj}) =
\begin{pmatrix} E_{jn}\otimes B_{jn}+
E_{nj}\otimes B_{jn}^* & 0 \cr 0 & 0 \cr\end{pmatrix}, \
\Phi(E_{jn}-E_{nj}) = \begin{pmatrix} E_{jn}\otimes C_{jn}-
E_{nj} \otimes C_{jn}^*  & 0 \cr 0 & 0 \cr\end{pmatrix}.$$
Here $B_{jn},C_{jn} \in \bM_k$ and $D^{-1}B_{jn},  D^{-1}C_{jn} \in \bU_k$ commute with $D$.

Because every matrix in the range of the map $\Phi$ has its
last $r- nk$ rows
and last $s-nk$ columns equal to zero, we will assume that
$r = nk$ and $s = nk$ for simplicity (by removing the last $r-nk$ rows and
$s-nk$ columns from every matrix in the range space).
Let $\{e_1, \dots, e_n\}$ be the standard basis for $\IC^n$.
For $j = 2, \dots, n-1$, consider the disjoint pair
$$X_1 = (e_1 + e_j + e_n)(e_1+e_j+e_n)^t
\quad \hbox{ and } \quad
X_2 = (2e_1-e_j-e_n)(2e_1-e_j-e_n)^t.$$
Then
$\Phi(X_1)$ and $\Phi(X_2)$ are disjoint.
If we partition $\Phi(X_1), \Phi(X_2)$ as $n\times n$ block matrices
$Z = (Z_{ij})_{1 \le i, j \le n}$
such that each block is in $\bM_k$, then all the blocks are
zero except for the $(p,q)$ blocks with $p,q\in \{1,j,n\}$.
Deleting all the zero blocks, we get the following two
$3\times 3$ block matrices.
$$Z_1 =
\begin{pmatrix} D & D & B_{1n} \cr D & D & B_{jn} \cr  B_{1n}^* & B_{jn}^* & D\cr
\end{pmatrix}
\quad \hbox{ and } \quad
Z_2 =
\begin{pmatrix} 4 D & -2 D & -2 B_{1n} \cr
-2 D & D & B_{jn} \cr -2 B_{1n}^* & B_{jn}^* & D\cr
\end{pmatrix}.$$
Both the $(1,1)$ and $(1,2)$ blocks of $Z_1Z_2^*$ equal $0_k$,
i.e.,
$$0_k = 2D^2 - 2B_{1n}B_{1n}^* = -D^2 + B_{1n}B_{jn}^*.$$
We see that
$B_{1n}B_{1n}^* = D^2 = B_{1n}B_{jn}^*$.
Since $B_{1n}$ is the product of $D$ and a unitary
 matrix, it is invertible.
So,  $B_{1n} = B_{jn}$ for $j = 2, \dots, n-1$.

Similarly, we can consider the disjoint pair
$$
X_3 = (e_1 + e_j + e_n)(-e_1-e_j+e_n)^t\quad \text{and}\quad
X_4 = (e_1+e_j-2e_n)(e_1+e_j+2e_n)^t.
$$
Then removing the zero blocks of $\Phi(X_3)$ and $\Phi(X_4)$, we get
$$Z_3 =
\begin{pmatrix} -D & -D & C_{1n} \cr -D & -D & C_{jn}
\cr  -C_{1n}^* & - C_{jn}^* & D\cr
\end{pmatrix}
\quad \hbox{ and } \quad
Z_4 =
\begin{pmatrix} D &  D & 2 C_{1n} \cr
D & D & 2 C_{jn} \cr -2 C_{1n}^* & -2 C_{jn}^* & -4D\cr
\end{pmatrix}.$$
Both the $(1,1)$ and $(1,2)$ blocks of $Z_3Z_4^*$ equal $0_k$, i.e.,
$$0_k = -2D^2 + 2C_{1n}C_{1n}^* = -2D^2 + 2C_{1n} C_{jn}^*.$$
We see that $C_{1n}C_{1n}^* = D^2= C_{1n} C_{jn}^*$.
Since $C_{1n}$ is the product of $D$ and
a unitary (real orthogonal) matrix, it is invertible. Thus,
$C_{1n} = C_{jn}$ for $j = 2, \dots, n-1$.

Let $W$ be the unitary  matrix
$D^{-1} B_{1n}\in \bM_n.$ Replace $\Phi$ by the map
$X \mapsto (I_{(n-1)k} \oplus W)\Phi(X) (I_{(n-1)k} \oplus W^*)$.
Then with $\hat C = C_{jn} W^*$ for $j = 1, \dots, n-1$, we have
$$\Phi(E_{ij}+E_{ji})
=  (E_{ij}+E_{ji})\otimes D, \quad 1 \le i \le j \le n,$$
$$\Phi(E_{ij}-E_{ji}) =  (E_{ij}-E_{ji}) \otimes \hat D,
\quad 1 \le i < j \le n-1,$$
$$\Phi(E_{jn}-E_{nj}) = E_{jn}\otimes \hat C-E_{nj}\otimes {\hat C}^*, \quad j = 1, \dots, n-1.$$
Recall that $P$ is a permutation matrix such that
$D = P^t(Q_1\oplus Q_2)P$.
Now replace $\Phi$ by $X \mapsto (I_n \otimes P)\Phi(X)(I_n\otimes P^t)$.
Then
$$\Phi(E_{ij}+E_{ji}) =  (E_{ij}+E_{ji})\otimes (Q_1\oplus Q_2),
\quad 1 \le i \le j \le n,$$
$$\Phi(E_{ij}-E_{ji}) =  (E_{ij}-E_{ji}) \otimes (Q_1\oplus -Q_2),
\quad 1 \le i < j \le n-1,$$
$$\Phi(E_{jn}-E_{nj}) = E_{jn}\otimes G -E_{nj} \otimes G^*,
\quad j = 1, \dots, n-1,$$
where $G = P\hat C P^t$.

It remains to show that $G = Q_1 \oplus -Q_2$ so that
$E_{jn}\otimes G - E_{nj}\otimes G^* = (E_{jn}-E_{nj})\otimes (Q_1 \oplus -Q_2)$.
To this end,
consider the disjoint pair
$X_5 = E_{22}+E_{nn} - E_{2n}-E_{n2}$ and $X_6 = E_{12}+E_{1n}- E_{21} - E_{n1}$.
Then $Z_5 = \Phi(X_5)$ and $Z_6 = \Phi(X_6)$ are disjoint.
If we partition $\Phi(X_5), \Phi(X_6)$ as $n\times n$ block matrices
$Z = (Z_{ij})_{1 \le i, j \le n}$
such that each block is in $\bM_k$, then all the blocks are
zero except for the $(p,q)$ blocks with $p,q\in \{1,2,n\}$.
Let $Q = Q_1 \oplus Q_2$ and $C_{12} = Q_1 \oplus - Q_2$.
Deleting all the zero blocks, we get the following two matrices.
$$Z_5 =
\begin{pmatrix} 0_k & 0_k & 0_k \cr 0_k & Q & -Q \cr  0_k & -Q & Q\cr
\end{pmatrix}
\quad \hbox{ and } \quad
Z_6 =
\begin{pmatrix} 0_k & C_{12} & G \cr
-C_{12}^* & 0_k & 0_k \cr -G^* & 0_k & 0_k\cr
\end{pmatrix}.$$
Now, the $(1,2)$ block of $Z_6 Z_5^*$ is zero, i.e.,
$C_{12}Q = GQ$. It follows that $G = C_{12} = Q_1 \oplus -Q_2$.
Thus, the desired result follows. \end{proof}

To prove the theorem when the domain is $\bM_{m,n}$ with $m < n$,
we can apply the result for the restriction of $\Phi$ to the
subspace spanned by $\{E_{ij}: 1 \le i, j \le m\}$ and assume the
restriction map has nice structure. Then we have to show that $\Phi(E_{il})$
also has a nice form for $l > m$. To do that we need another technical lemma
showing that if $\Phi(E_{ij})$ and $\Phi(E_{kj})$ have nice forms,
then $\Phi(E_{il})$ and $\Phi(E_{kl})$ also have nice forms.
We  state and prove the results for a special case in the following, in view of Lemma \ref{lem:M2all}.

\begin{lemma} \label{lem3}
Let $Q_1\in \bM_{k_1}, Q_2\in \bM_{k_2}$ with $k_1+k_2 = k$
be diagonal matrices with positive diagonal entries arranged in descending order.
Let $\Phi: \bM_2 \rightarrow \bM_{r,s}$ be a nonzero linear map preserving disjointness.
\begin{itemize}
\item[{\rm (a)}] Assume
$$\Phi(E_{11}) = \begin{pmatrix}
Q_1 & 0 & 0 & 0 \cr
0 & 0_{k_1,k_2} & 0 & 0 \cr
0 & Q_2 & 0_{k_2} & 0 \cr
0 & 0 & 0 & 0_{r_1, s_1}\cr\end{pmatrix},
\quad
\Phi(E_{21}) = \begin{pmatrix}
0_{k_1}& 0 & 0 & 0 \cr
Q_1 & 0_{k_1,k_2} & 0 & 0 \cr
0 & 0_{k_2} &Q_2 & 0 \cr
0 & 0 & 0 & 0_{r_1,s_1}\cr\end{pmatrix},$$
where $(r_1,s_1) = (r-2k_1-k_2, s - k_1-2k_2)$.
Then there exist $R = R_1 \oplus R_2 \in \bU_{k_1} \oplus \bU_{k_2}$,
$U \in \bU_{r-k}, V \in \bU_{s-k}$ such that
\begin{gather*}
R_1^*Q_1 R_1 = Q_1,\quad R_2^* Q_2 R_2 = Q_2,\\
U \begin{pmatrix} Q_1 \cr 0_{r-k-k_1,k_1} \cr\end{pmatrix} R_1 =
\begin{pmatrix} Q_1 \cr 0_{r-k-k_1,k_1}\cr\end{pmatrix},
\intertext{and}
R_2^* ( Q_2 \ | \ 0_{k_2,s-k-k_2}) V = (Q_2 \ | \ 0_{k_2, s-k-k_2});
\end{gather*}
moreover, if $U = \begin{pmatrix} U_{11} & U_{12} \cr U_{21} & U_{22} \cr
\end{pmatrix}$ with $U_{11} \in \bM_{k_1}$, then the modified map  $\Psi$ defined by
$$X \mapsto
\begin{pmatrix} R_1^* & 0 & 0 & 0\cr
0 & U_{11} & 0 & U_{12} \cr
0 & 0 & R_2^* & 0 \cr
0 & U_{21} & 0 & U_{22}\cr\end{pmatrix}
\Phi(X) \begin{pmatrix} R_1 & 0 & 0 \cr 0 & R_2 & 0 \cr 0 & 0 & V\cr\end{pmatrix}
$$
satisfies
\begin{gather*}
\Psi(E_{11}) =  \Phi(E_{11}),\quad \Psi(E_{21}) = \Phi(E_{21}),\\
\Psi(E_{12}) =
\begin{pmatrix}
0_{k_1} & 0 & 0 & Q_1 & 0 \cr
0 & 0_{k_1,k_2} & 0 & 0 & 0\cr
0 & 0 & 0_{k_2} & 0 & 0\cr
0 & Q_2 & 0 & 0 & 0\cr
0 & 0 & 0 & 0 & 0_{r-2k,s-2k}\cr\end{pmatrix},
\
\Psi(E_{22}) =
\begin{pmatrix}
0_{k_1} & 0 & 0 & 0 & 0 \cr
0 & 0_{k_1,k_2} & 0 & Q_1 & 0\cr
0 & 0 & 0_{k_2} & 0 & 0\cr
0 & 0 & Q_2 & 0 & 0\cr
0 & 0 & 0 & 0 & 0_{r-2k,s-2k}\cr\end{pmatrix}.
\end{gather*}
Consequently, before the modification we have
$$\Phi(E_{12}) =
\begin{pmatrix}
0_{k_1} & 0 & 0 & \hat Y_{1} \cr
0  & 0_{k_1,k_2} & 0 & 0\cr
0 & 0 & 0_{k_2} & 0 \cr
0 & \hat Y_{2} & 0 & 0_{r_1,s_1} \cr \end{pmatrix},
\  \Phi(E_{22}) =
\begin{pmatrix}
0_{k_1} & 0 & 0 & 0  \cr
0  & 0_{k_1,k_2} & 0 & \hat Z_{1} \cr
0 & 0 & 0_{k_2} & 0 \cr
0 & 0 &  \hat Z_{2} & 0_{r_1,s_1} \cr \end{pmatrix},
$$
where $\hat Y_1, \hat Z_1$ have singular values equal to the diagonal entries of
$Q_1$, and $\hat Y_2, \hat Z_2$ have singular values equal to the diagonal entries
of $Q_2$.
\item[{\rm (b)}]
Suppose
\begin{equation} \label{lem3eq}
\Phi(E_{ij}) = \begin{pmatrix} E_{ij} \otimes Q_1
& 0 & 0 \cr
0 & E_{ji}\otimes Q_2 & 0  \cr
0 & 0 & 0_{r_2, s_2}\cr\end{pmatrix} \quad \hbox{ for }
(i,j) \in \{(1,1), (2,1), (2,2)\},
\end{equation}
 and $(r_2, s_2) = (r-2k,s-2k)$.
Then $\Phi(E_{12})$ also satisfies {\rm (\ref{lem3eq})}.
\end{itemize}
\end{lemma}

\noindent
\begin{proof}\rm (a)
By Lemma \ref{lem1}, we know that
the disjoint matrices $\Phi(E_{22})$ and $\Phi(E_{11})$ have the same rank. So, $r, s \ge 2k$.
Let $P_1 \in \bM_{2k}$ be a permutation matrix such that
$[X_1|X_2|X_3|X_4]P_1 = [X_1|X_3|X_2|X_4]$
whenever $X_1, X_2 \in \bM_{2k,k_1}$
and $X_3, X_4 \in \bM_{2k,k_2}$.
Then the map $\hat \Phi$ defined by
$\hat \Phi(X) = (P_1^t \oplus I_{r-2k})\Phi(X)$ will still preserve disjointness such that
$\hat\Phi(E_{11})$ and $\hat \Phi(E_{21})$  equal
$$\hat \Phi(E_{11}) =
\begin{pmatrix}
Q_1 & 0 & 0 & 0  \cr
0 & Q_2 & 0 & 0 \cr
0 & 0 & 0_{k_1,k_2}  & 0 \cr
0 & 0 & 0 & 0_{r_1,s_1} \end{pmatrix} \quad \hbox{ and } \quad
\hat \Phi(E_{21}) =
\begin{pmatrix} 0_{k_1}  &  0 & 0 & 0\cr
0 & 0_{k_2} &  Q_2 & 0  \cr
Q_1 & 0 & 0_{k_1,k_2} & 0 \cr
0 & 0 & 0 &  0_{r_1, s_1} \cr
\end{pmatrix}.$$
Suppose $P_2 \in \bM_k$ is a permutation matrix such that
$D_1 = P_2^t(Q_1\oplus Q_2)P_2$ has diagonal entries
arranged in descending order.
We can then find   $U_1 \in\bU_{r-k}$
and $V_1 \in \bU_{s-k}$
such that
$$(P_2^t \oplus U_1)\hat \Phi(E_{22}) (P_2 \oplus V_1)
= \begin{pmatrix} 0_k & 0 & 0 \cr
0 & D_2 & 0 \cr
0 & 0 & 0 \cr \end{pmatrix},$$
where $D_2$ is a diagonal matrix with positive diagonal entries arranged in descending order.

Applying Lemma \ref{lem1}, we can find
$S_2 \in \bU_k, U_2 \in \bU_{r-k}, V_2 \in \bU_{s-k}$ such that
the map $\Psi_1$ defined by
$$X\mapsto (S_2^*\oplus U_2) (P_2^t \oplus U_1) \hat \Phi(X)
(P_2 \oplus V_1)(S_2 \oplus V_2)$$
satisfies
$$E_{ij} \mapsto
(E_{ij} \otimes (\hat Q_1 \oplus 0_{\ell_2}) + E_{ji}\otimes (0_{\ell_1} \oplus \hat Q_2)), \quad 1 \le i, j \le 2,$$
where $\hat Q_1 \in \bM_{\ell_1}$ and $\hat Q_2 \in \bM_{\ell_2}$
are diagonal matrices with positive diagonal entries arranged in descending order. Let $\Psi$ be defined by $\Psi(X)=\Psi_1(X)(I_k\oplus P_3\oplus I_{s-2k})$, where $P_3 \in \bM_k$ is a permutation matrix such that
$[X_1|X_2]P_3 = [X_2|X_1]$ whenever $X_1\in \bM_{k,k_1}$
and $X_2\in \bM_{k,k_2}$. Then the map $\Psi$ satisfies
$$E_{ij} \mapsto (E_{ij} \otimes (\hat Q_1 \oplus 0_{\ell_2}) + E_{ji}\otimes (0_{\ell_1} \oplus \hat Q_2))(I_k\oplus P_3\oplus I_{s-2k}), \quad 1 \le i, j \le 2.$$
Let $R = P_2S_2 \in \bM_k$, $V  = V_1V_2(P_3\oplus I_{s-2k}) \in \bM_{s-k}$, and $U = U_2 U_1 \in \bM_{r-k}$.
Then
$$\Psi(X) = (R^* \oplus U)\hat \Phi(X) (R \oplus V) \quad \hbox{ for all } X\in \bM_2.
$$
If we partition
$\Psi(X)$ into a $2\times 2$ block matrix such that the $(1,1)$ block
lies in $\bM_k$, then
the diagonal entries of $\hat Q_1$ are the singular values
of the $(2,1)$ block of $\hat \Phi(E_{21})$ (using the same
partition). So, $\hat Q_1 = Q_1$ and $\hat Q_2 = Q_2$.
Hence, $\hat \Phi(E_{21}) = \Psi(E_{21})$. It follows that
\begin{equation}
\label{eq-r}
R^* \begin{pmatrix}
0_{k_1, k_2} & 0_{k_1,s_1} \cr
Q_2 & 0_{k_2,s_1} \cr\end{pmatrix}V =
\begin{pmatrix} 0_{k_1, k_2} & 0_{k_1,s_1} \cr
Q_2 & 0_{k_2,s_1} \cr\end{pmatrix},  \ \
U\begin{pmatrix} Q_1 & 0_{k_1,k_2}\cr 0_{r_1,k_1} & 0_{r_1,k_2} \cr\end{pmatrix}R
= \begin{pmatrix} Q_1 & 0_{k_1,k_2}\cr 0_{r_1,k_1} & 0_{r_1,k_2}
\cr\end{pmatrix}.
\end{equation}
As a result,
$$R^* \begin{pmatrix} 0_{k_1} & 0 \cr 0 & Q_2^2\cr \end{pmatrix}  R =
\begin{pmatrix} 0_{k_1} & 0 \cr 0 & Q_2^2\cr \end{pmatrix} \quad \hbox{ and }
\quad
R^* \begin{pmatrix} Q_1^2 & 0 \cr 0 & 0_{k_2}\cr \end{pmatrix}R
=
\begin{pmatrix} Q_1^2 & 0 \cr 0 & 0_{k_2}\cr \end{pmatrix} .$$
Thus, $R = R_1 \oplus R_2$ with $R_1 \in \bM_{k_1}, R_2 \in \bM_{k_2}$.
Since $Q_1$ and $Q_2$ are diagonal matrices with positive diagonal entries,
we see that $R_1^*Q_1R_1 = Q_1$ and $R_2^*Q_2R_2 = Q_2$.
Moreover, by (\ref{eq-r}) we have
$$U \begin{pmatrix} Q_1 \cr 0_{r-k-k_1,k_1} \cr\end{pmatrix} R_1 =
\begin{pmatrix} Q_1 \cr 0_{r-k-k_1,k_1}\cr\end{pmatrix}
\ \hbox{ and } \
R_2^* ( Q_2 \ | \ 0_{k_2,s-k-k_2}) V = (Q_2 \ | \ 0_{k_2, s-k-k_2}).$$
One can then check that the modified map $\hat\Psi(X)=(P_1^t\oplus I_{r-2k})\Psi(X)$ has the desired property.

Now, we turn to $\Phi(E_{12})$ and $\Phi(E_{21})$.
If $U = (U_{ij})_{1 \le i, j \le 3}\in \bM_{r-k}$ with
$U_{11} \in \bM_{k_1}, U_{22} \in \bM_{k_2}$, and
$V = (V_{ij})_{1 \le i, j \le 3} \in \bM_{s-k}$ with $V_{11} \in \bM_{k_2},
 V_{22} \in \bM_{k_1}$, then
$$\hat \Phi(E_{12})
= (R\oplus U^*)\Psi(E_{12})(R^*\oplus V^*)
= \begin{pmatrix} 0_k & F_{12} \cr
F_{21} & 0_{r-k,s-k} \end{pmatrix},$$
where
$$F_{12} =
R\begin{pmatrix} 0 & Q_1 & 0 \cr 0_{k_2} & 0_{k_2, k_1} & 0_{k_2,s-2k}
\cr\end{pmatrix} V^*
= \begin{pmatrix}
R_1Q_1 V_{12}^* & R_1Q_1 V_{22}^* & R_1 Q_1 V_{32}^*\cr
0_{k_2} & 0_{k_2,k_1} & 0_{k_2, s-2k} \cr\end{pmatrix},$$
$$
F_{21} = U^* \begin{pmatrix} 0_{k_1} & 0 \cr 0 & Q_2 \cr 0 & 0 \cr\end{pmatrix}
R^* = \begin{pmatrix} 0_{k_1} & U_{21}^*Q_2R_2^* \cr
0_{k_2,k_1} & U_{22}^*   Q_2 R_2^*\cr
0_{r-2k,k_1} & U_{23}^* Q_2 R_2^*\cr
\end{pmatrix}.$$
Note that $\hat \Phi(E_{12})$ and $\hat \Phi(E_{21})$ are disjoint.
So, $U_{21}^* Q_2 R_2^*,
R_1Q_1V_{12}^*\in \bM_{k_1,k_2}$ are zero blocks.
Since $R_1Q_1$ and $Q_2R_2^*$ are invertible, we see that
\begin{equation} \label{V11U21}
U_{21}^* = 0_{k_1,k_2} \quad \hbox{ and } \quad V_{12}^* = 0_{k_1,k_2}.
\end{equation}
As a result, $\Phi(E_{12}) = (P_1\oplus I_{r-2k})\hat \Phi(E_{12})$ has the asserted form
with
$$\hat Y_{1} =
(R_1Q_1 V_{22}^* \ | \  R_1 Q_1 V_{32}^*)\quad \hbox{ and } \quad
\hat Y_{2} =
\begin{pmatrix} U_{22}^*   Q_2 R_2^*\cr  U_{23}^* Q_2 R_2^*\cr
\end{pmatrix}.$$
Also, $\hat \Phi(E_{22}) =
\begin{pmatrix}0_k & 0 \cr 0 & G\cr\end{pmatrix}$ with
\begin{eqnarray*}
G &=& U^*\begin{pmatrix}
0 & Q_1 & 0 \cr Q_2 & 0 & 0 \cr 0 & 0 & 0_{r-2k,s-2k} \cr
\end{pmatrix}V^* \\
&=& U^*\begin{pmatrix}
0 & Q_1 & 0 \cr 0_{k_2} & 0 & 0 \cr 0 & 0 & 0_{r-2k,s-2k} \cr
\end{pmatrix}V^* +
U^*\begin{pmatrix}
0 & 0_{k_1} & 0 \cr Q_2 & 0 & 0 \cr 0 & 0 & 0_{r-2k,s-2k} \cr
\end{pmatrix}V^*
\\
&=& U^*\begin{pmatrix}
0 & Q_1  \cr 0_{k_2} & 0  \cr 0 & 0 & \cr
\end{pmatrix}
R^{'*}R'   \begin{pmatrix}
V_{11}^* & V_{21}^* & V_{31}^*  \cr V_{12}^* & V_{22}^* & V_{32}^* \cr
\end{pmatrix} +
\begin{pmatrix}
U_{11}^* & U_{21}^*  \cr U_{12}^* & U_{22}^*  \cr U_{13}^* & U_{23}^*  \cr
\end{pmatrix}
R^* R \begin{pmatrix}
0 & 0_{k_1} & 0 \cr Q_2 & 0 & 0 \cr
\end{pmatrix}V^* \\
&=& \begin{pmatrix}
0 & Q_1  \cr 0_{k_2} & 0  \cr 0 & 0 & \cr
\end{pmatrix}
R' \begin{pmatrix}
V_{11}^* & V_{21}^* & V_{31}^*  \cr V_{12}^* & V_{22}^* & V_{32}^* \cr
\end{pmatrix} +
 \begin{pmatrix}
U_{11}^* & U_{21}^*  \cr U_{12}^* & U_{22}^*  \cr U_{13}^* & U_{23}^*  \cr
\end{pmatrix} R^* \begin{pmatrix}
0_{k_1,k_2} & 0 & 0 \cr Q_2 & 0 & 0 \cr
\end{pmatrix}
\end{eqnarray*}
by (\ref{eq-r}), where $R'=R_2\oplus R_1$.
Thus,  by (\ref{V11U21}), we have
$$G = \begin{pmatrix} Q_1R_1V_{12}^* + U_{21}^*R_2^*Q_2
& Q_1R_1V_{22}^* & Q_1 R_1V_{32}^* \cr
U_{22}^* R_2^* Q_2 & 0 & 0 \cr
U_{23}^* R_2^* Q_2 & 0 & 0 \cr
\end{pmatrix} =
\begin{pmatrix} 0_{k_1,k_2}
& Q_1R_1V_{22}^* & Q_1 R_1V_{32}^* \cr
U_{22}^* R_2^* Q_2 & 0 & 0 \cr
U_{23}^* R_2^* Q_2 & 0 & 0 \cr
\end{pmatrix}.
$$
As a result, $\Phi(E_{22}) = (P_1\oplus I_{r-2k})\hat\Phi(E_{22})$ has the asserted form
with
$$\hat Z_{1} = (Q_1R_1V_{22}^* \ | \ Q_1 R_1V_{32}^*) \quad \hbox{ and } \quad
\hat Z_{2} =
\begin{pmatrix}
U_{22}^* R_2^* Q_2 \cr
U_{23}^* R_2^* Q_2 \cr
\end{pmatrix}.
$$

(b)
Applying a block permutation, we may assume that
$\Phi(E_{11})$,
$\Phi(E_{21})$, $\Phi(E_{22})$ equal
$$\begin{pmatrix} Q_1 \oplus Q_2  & 0_k & 0  \cr
0_k & 0_k  & 0 \cr
0 & 0 & 0_{\hat r,\hat s} \end{pmatrix}, \quad
\begin{pmatrix} 0_{k}  & 0_{k_1} \oplus Q_2 & 0  \cr
Q_1 \oplus 0_{k_2} & 0_k  & 0 \cr
0 & 0 & 0_{\hat r, \hat s} \cr
\end{pmatrix}, \quad
\begin{pmatrix} 0_k  & 0_k & 0  \cr 0_k & Q_1 \oplus Q_2  & 0 \cr
0 & 0 & 0_{\hat r, \hat s} \cr \end{pmatrix},$$
respectively. We need to show that
$$\Phi(E_{12}) =
\begin{pmatrix} 0_{k}  & Q_1 \oplus 0_{k_2} & 0  \cr
0_{k_1} \oplus Q_2  & 0_k  & 0 \cr
0 & 0 & 0_{\hat r, \hat s}
\end{pmatrix}.
$$

Suppose $\hat P \in \bM_k$ is a permutation
matrix such that $\hat D = \hat P^t(Q_1 \oplus Q_2)\hat P$ is a diagonal
matrix with entries in descending order. Applying Lemma \ref{lem1} to
the map
$$X \mapsto (\hat P^t \oplus \hat P^t \oplus I_{r-2k})
\Phi(X)(\hat P \oplus \hat P \oplus I_{s-2k}),$$
we conclude that there exist a permutation $P \in \bM_k$
and $W_1, W_2 \in \bU_k$ commuting with $\hat D$ such that
for $W = \hat P W_1 W_2 P \oplus \hat P W_2 P \in \bM_{2k}$,
the map $\Psi$ defined by
$X \mapsto  (W^*\oplus I_{r-2k})\Phi(X)(W\oplus I_{s-2k})$
has the form
$$E_{ij} \mapsto
E_{ij} \otimes (\hat Q_1 \oplus 0_{\ell_2}) + E_{ji}\otimes (0_{\ell_1} \oplus \hat Q_2), \quad 1 \le i, j \le 2,$$
where $\hat Q_1 \in \bM_{\ell_1}$ and $\hat Q_2 \in \bM_{\ell_2}$
are diagonal matrices with positive diagonal entries arranged in descending order.
Note that the diagonal entries of $\hat Q_1$ are the singular values
of the $(1,2)$ block of $\Phi(E_{12})$. So, $\hat Q_1 = Q_1$ and $\hat Q_2 = Q_2$.
Consequently,
$$\Phi(X) =  (W\oplus I_{r-2k})\Psi(X)(W^*\oplus I_{s-2k})$$
has the asserted form. \end{proof}

\begin{proof}[Proof of Theorem \ref{thm1}]
Without loss of generality, we assume $2 \le m\le n$.
We prove the result by induction on $n-m$.
If $n-m = 0$, the result follows from Lemma \ref{lem2}.
Suppose $n-m=\ell \geq 1$ and the result holds for the cases
when $n-m < \ell$.

By the induction assumption on the restriction map of $\Phi$ on the span of $\cC_n = \{E_{ij}: 1 \le i\le m, 1 \le j < n\}$, there are diagonal matrices $Q_1\in \bM_{k_1}, Q_2\in \bM_{k_2}$ with positive entries arranged in descending order, and $U_1 \in \bU_r, V_1 \in \bU_s$
such that the map $U_1\Phi(X)V_1$ satisfies
\begin{equation}\label{case3}
E_{ij} \mapsto \begin{pmatrix} \hat E_{ij} \otimes Q_1 & 0 & 0\cr
0 &  \hat E_{ji} \otimes Q_2  & 0 \cr
0 & 0 & 0_{\hat r, \hat s}
\cr\end{pmatrix} \qquad \hbox{ for all } E_{ij} \in \cC_n,
\end{equation}
where
$\{E_{ij}: 1 \le i \le m, 1 \le j \le n\}$
is the standard basis for $\bM_{m,n}$,
$\{\hat E_{ij}: 1 \le i \le m, 1 \le j < n\}$
is the standard basis for $\bM_{m,n-1}$, and
$(\hat r, \hat s) = (r - mk_1 - (n-1)k_2, s - (n-1)k_1-mk_2)$.
For notational simplicity, we assume that $U_1 = I_r, V_1 = I_s$.

Consider the restriction of $\Phi$ on
$\operatorname{span}\{E_{ij}, E_{in}, E_{mj}, E_{mn}\}$ for all
$1 \le i < m, 1 \le j < n$.
By Lemma \ref{lem3} (a), we see that
\begin{equation}\label{Emn}
\Phi(E_{mn}) = \begin{pmatrix} 0_{mk_1,(n-1)k_1} & 0 & Z_{1} \cr
0 & 0_{(n-1)k_2,mk_2} & 0 \cr
0 & Z_{2} & 0_{\hat r, \hat s} \cr\end{pmatrix},
\end{equation}
where only the last $k_1$ rows of $Z_{1}$ can be
nonzero, and only the last $k_2$ columns of $Z_{2}$ can be nonzero.

Similarly,
\begin{equation} \label{E1n}
\Phi(E_{1n}) = \begin{pmatrix} 0_{mk_1,(n-1)k_1} & 0  & Y_{1} \cr
0 & 0_{(n-1)k_2,mk_1} & 0 \cr
0 & Y_{2} & 0_{\hat r,\hat s} \cr\end{pmatrix}
\end{equation}
where only the first $k_1$ rows of $Y_{1}$ can be
nonzero, and only the first $k_2$ columns of $Y_{2}$ can be nonzero.

Now, consider the restriction of $\Phi$
on $\operatorname{span}\{E_{11}, E_{1n}, E_{m1}, E_{mn}\}.$
By Lemma \ref{lem3} (a), there exist
$R = R_1 \oplus R_2 \in \bU_{k_1} \oplus \bU_{k_2}$,
$U \in \bU_{\hat r}$ and
$V \in \bU_{\hat s}$
such that
\begin{gather*}
R_1^*Q_1 R_1 = Q_1,\quad
R_2^* Q_2 R_2 = Q_2,\\
U \begin{pmatrix} Q_1 \cr 0_{\hat r -k_1,k_1} \cr\end{pmatrix} R_1 =
\begin{pmatrix} Q_1 \cr 0_{\hat r-k_1,k_1}\cr\end{pmatrix},
\ \hbox{ and } \
R_2^* ( Q_2 \ | \ 0_{k_2,\hat s-k_2}) V = (Q_2 \ | \ 0_{k_2, \hat s-k_2});
\end{gather*}
moreover, if $U = \begin{pmatrix} U_{11} & U_{12} \cr U_{21} & U_{22}\cr
\end{pmatrix}$ with $U_{11} \in \bM_{k_1}$, then
$$\begin{pmatrix} R_1^* & 0 & 0 & 0\cr
0 & U_{11} & 0 & U_{12} \cr
0 & 0 & R_2^* & 0 \cr
0 & U_{21} & 0 & U_{22}\cr\end{pmatrix}
\begin{pmatrix}
0_{2k_1,k_1} &  0_{2k_1,2k_2} & Z_{1}\cr
0 & 0_{k_2,2k_2} &  0 \cr
0 & Z_{2} & 0_{\hat r, \hat s}\cr
\end{pmatrix} \begin{pmatrix} R_1 & 0 & 0 \cr 0
& R_2 & 0 \cr 0 & 0 & V\cr\end{pmatrix}$$
$$= \begin{pmatrix}
0_{k_1}  &  0 & 0 & 0 & 0 \cr
0 &  0_{k_1,k_2} & 0 & Q_1 & 0 \cr
0 & 0 & 0_{k_2} & 0 & 0\cr
0 & 0 & Q_2 & 0_{k_2,k_1} & 0  \cr
0 & 0 & 0 & 0 & 0_{r_1,s_1},
\end{pmatrix}, \hskip 2in \
$$
where $(r_1,s_1) = (\hat r-2k_1, \hat s-2k_2)$.
Consequently,
the modified map  $\Psi$ defined by
$$X \mapsto
\begin{pmatrix} I_{m-1}\otimes R_1^* & 0 & 0 & 0\cr
0 & U_{11} & 0 & U_{12} \cr
0 & 0 & I_{n-1}\otimes R_2^* & 0 \cr
0 & U_{21} & 0 & U_{22}\cr\end{pmatrix}
\Phi(X) \begin{pmatrix} I_{n-1}\otimes R_1 & 0 & 0 \cr 0
& I_{m-1} \otimes R_2 & 0 \cr 0 & 0 & V\cr\end{pmatrix}
$$
satisfies $\Psi(E_{ij}) =  \Phi(E_{ij})$ for all
$1 \le i \le m, 1 \le j < n-1$,  and
$\Psi(E_{mn})$ has the form (\ref{Emn})
with
$$Z_{1} = \begin{pmatrix} 0 & 0 \cr Q_1 & 0 \cr\end{pmatrix}
\ \hbox{ and } \
Z_{2} = \begin{pmatrix} 0 & Q_2 \cr 0 & 0 \cr\end{pmatrix}.$$
Let $\tilde P \in \bM_s$ be the permutation matrix satisfying
$[X_1|X_2|X_3|X_4]\tilde P = [X_1|X_3|X_2|X_4]$ whenever
$X_1 \in \bM_{r,(n-1)k_1}$, $X_2 \in \bM_{r, mk_2}$,
$X_3 \in \bM_{r,k_1}$, $X_4 \in \bM_{r, \hat s - k_1}$.
Then the map $\hat \Psi$ defined by
$X \mapsto  \Psi(X) \tilde P$
satisfies
\begin{equation} \label{final}
\hat \Psi(E_{ij}) =
\begin{pmatrix}
E_{ij}\otimes Q_1 & 0 & 0 \cr
0 & E_{ji}\otimes Q_2 & 0 \cr
0 & 0 & 0_{\hat r - k_2, \hat s - k_1}\cr\end{pmatrix}
\end{equation}
for $(i,j) \in \{ (u,v): 1\le u \le m, 1 \le v < n\} \cup \{(m,n)\}$.
For $j = 2, \dots, n-1$, consider the
restriction of $\Psi$ on $\operatorname{span}\{ E_{jj}, E_{jn}, E_{mj}, E_{mn}\}$.
Thus, $\hat \Psi(E_{jj}), \hat \Psi(E_{mj}), \hat \Psi(E_{mn})$
have the form (\ref{final}), and so must $\hat \Psi(E_{jn})$ by Lemma \ref{lem3} (b).
As a result, $\hat \Psi(E_{ij})$ has the form in
(\ref{final}) for all $1 \le i \le m,
1 \le j \le n$.
\end{proof}

\section{Nonsurjective (zero) Triple Product Preservers and
JB*-homomorphisms on rectangular matrices}\label{sec:3}

Notice that the set $\bM_n(\mathbb{C})$ of complex
square matrices is a C$^*$-algebra.
Let $T:\mathcal{A}\to\mathcal{B}$ be a bounded linear map between
C$^*$-algebras. In \cite[Theorem 3.2]{W05}, it was shown that $T$ is a triple homomorphism with respect to the Jordan triple product,
$$
\{a, b, c\}=\frac{1}{2}(ab^*c+cb^*a)\quad \text{for all}\   a, b, c\in\mathcal{A},
$$
if and only if $T$ preserves disjointness and $T^{**}(1)$ is a partial isometry in $B^{**}$.
In the case that $T$ is surjective, the condition on $T^{**}(1)$ can be dropped as shown in \cite[Theorem 2.2]{LW13},
see also \cite{LCLWa}.
In \cite{BFGMP08}, on the other hand, it is
obtained a characterization of linear maps from $C^*$-algebras into
JB*-triples that preserve disjointness with some conditions.

In the following, we consider  the Jordan triple product
$\{A,B,C\} = \frac{1}{2} (AB^*C+CB^*A)$ of real or complex matrices $A,B,C \in \bM_{m,n}$.
A (real or complex) linear map $\Psi: \bM_{m,n}\to \bM_{r,s}$ between rectangular matrices
is called a JB*-\emph{triple homomorphism} if
\begin{equation}\label{jb-hom}
    \Psi(AB^*C + CB^*A)= \Psi(A)\Psi(B)^*\Psi(C) + \Psi(C)\Psi(B)^*\Psi(A)\quad \hbox{ for all } A,B,C \in \bM_{m,n}.
\end{equation}
We have the polarization identity
$$
2\{A,B,C\} = \{A+C, B, A+C\} -\{A,B,A\}-\{C,B,C\}\quad \text{for all}\  A,B,C\in \bM_{m,n}.
$$
In the complex case, letting the \emph{cube} $A^{(3)}= AA^*A$, we have
$$
4\{A,B,A\} = (B+A)^{(3)} + (B-A)^{(3)} - (B+\mathit{i}A)^{(3)} - (B-\mathit{i}A)^{(3)}\quad \text{for all}\  A,B\in \bM_{m,n}.
$$
Therefore, a linear map $\Phi$ between rectangular matrices is a JB*-triple homomorphism exactly when
$\Phi(AB^*A)=\Phi(A)\Phi(B)^*\Phi(A)$, and in the  complex case exactly when
$\Phi(AA^*A)=\Phi(A)\Phi(A)^*\Phi(A)$, for all $A, B \in \bM_{m,n}$.

We say that the matrix triple $(A, B, C)$ in $\bM_{m,n}$ has
\emph{zero triple product} if $\{A,B,C\} = 0_{m,n}$. A linear map $\Phi: \bM_{m,n} \rightarrow \bM_{r,s}$ preserves
zero triple products if
$$
\{A,B,C\}   =0_{m,n} \ \implies\ \{\Phi(A),\Phi(B),\Phi(C)\} = 0_{r,s}
\quad \text{for all}\ A,B,C \in \bM_{m,n}.
$$
For more information of JB*-triples, see, e.g., \cite{Chubook}.

We have the following result concerning the zero triple product
preservers and JB*-triple homomorphisms on rectangular matrices.

\begin{theorem}\label{3.1}
Let
$\Phi:\bM_{m,n}\to \bM_{r,s}$ be a  linear map.
\begin{enumerate}[(a)]
\item
$\Phi$ preserves zero triple products if and only if
there are  $U \in \bU_r, V\in \bU_s$,
and diagonal matrices $Q_1, Q_2$
with positive diagonal entries such that
\begin{equation}\label{s-form-2}
\Phi(A) =  U\begin{pmatrix}
A \otimes Q_1  & 0 & 0 \cr
0 & A^t \otimes Q_2 & 0 \cr
0 & 0 & 0 \cr\end{pmatrix}V.
\end{equation}
Here $Q_1$ or $Q_2$,
may be vacuous.

\item
$\Phi$ is a JB*-triple homomorphism if and only if
there exist $U \in \bU_r, V \in \bU_s$,
and nonnegative integers $q_1,q_2$ such that
\begin{equation}\label{s-form-3}
\Phi(A) = U \begin{pmatrix} A \otimes I_{q_1} & 0 & 0 \cr
0& A^t\otimes I_{q_2} &0\cr
0 & 0 & 0 \cr \end{pmatrix}V,
\end{equation}
where the size of the zero block at the bottom right corner is $(r-(q_1m+q_2n))\times (s-(q_1n+q_2m))$.
\end{enumerate}
\end{theorem}

To prove the above theorem, we need the
following lemma, which is valid for both real and complex matrices.
See  \cite[Lemma 1]{BFGMP08} for the complex case.
Recall that $A^*=A^t$ in the real case.

\begin{lemma}\label{lem:tcp}
Let $A,B\in \bM_{m,n}$.  The following conditions are equivalent to each other.
\begin{enumerate}[\rm (a)]
\item $A^*B=0_n$ and $AB^*=0_m$.
\item $AA^*B + BA^*A=0_{m,n}$.
\end{enumerate}
\end{lemma}
\begin{proof}
It suffices to prove (b)$\implies$(a).
Observe that from (b) we have
$$
0\leq (B^*A)(B^*A)^* = B^*AA^*B = - (B^*B)(A^*A).
$$
Taking adjoints of the Hermitian matrices, we have
$$
(B^*B)(A^*A)=  (A^*A)(B^*B).
$$
Therefore, the positive semi-definite $n\times n$ matrices $A^*A$ and $B^*B$ commute.
By spectral theory, the product $(B^*B)(A^*A)=-(B^*A)(B^*A)^*$ is also positive semi-definite, and thus
$B^*A=0$.
Similarly, we have $AB^*=0$.
\end{proof}

\begin{proof}[Proof of Theorem \ref{3.1}]
(a)
Suppose $\Phi$ preserves zero triple products. By Lemma \ref{lem:tcp}, if $A,B \in \bM_{m,n}$ are disjoint, then
$\Phi(A), \Phi(B) \in \bM_{r,s}$ are disjoint. So, $\Phi$ has the asserted form by Theorem \ref{thm1}.
The converse is clear.

(b) Suppose $\Phi$ is a JB*-triple homomorphism. Then it will preserve zero triple products, and thus by (a), be of the form \eqref{s-form-2}. Since
$E_{11}^{(3)} = E_{11}$, we have
$\Phi(E_{11})^{(3)} = \Phi(E_{11})$.
One gets the conclusions $Q_1=I_{q_1}$ and $Q_2=I_{q_2}$ as in \eqref{s-form-3}.
The converse is clear.
\end{proof}

Recall that a rectangular matrix $A$ is called a \emph{partial isometry}  if
$AA^*A=A$.  Equivalently, $A$ has singular values from the set $\{1,0\}$. We state our result using the complex notation. Of course, in the real case, we have  $X^* = X^t$, and a unitary matrix is a real orthogonal matrix.
It turns out that JB*-triple homomorphisms are closely related to linear preservers of  (disjoint)
partial isometries. Some assertions in the following might be known to experts, at least in the complex case.

\begin{theorem}\label{thm:tp}
Suppose $\Phi:\bM_{m,n}\to \bM_{r,s}$ is a linear map.
The following conditions are equivalent.
\begin{enumerate}[\rm (a)]

\item $\Phi$ maps partial isometries in
$\bM_{m,n}$ to partial isometries in $\bM_{r,s}$.

\item $\Phi$ sends disjoint (rank one) partial isometries to
disjoint partial isometries.

\item $\Phi$ preserves disjointness, and there is a nonzero partial isometry $P \in \bM_{m,n}$
such that $\Phi(P)$ is a partial isometry.

\item
$\Phi$ preserves matrix triples with zero JB*-triple product, and there is a nonzero partial isometry $P \in \bM_{m,n}$
such that $\Phi(P)$ is a partial isometry.

\item $\Phi$ is a JB*-triple homomorphism and has the form
{\rm (\ref{s-form-3})}.
\end{enumerate}
\end{theorem}

\begin{proof}\rm
The implication (e) $\implies$  (a) is clear.

(a) $\implies$ (b):  Let $A\in \bM_{m,n}$ be a rank one partial isometry,
and $\Phi(A) = U\begin{pmatrix} I_k & 0 \cr 0 & 0 \cr \end{pmatrix}V$,
where $U\in \bU_r, V\in \bU_s$.
Suppose $B \in \bM_{m,n}$ is a rank one partial
isometry disjoint from  $A$
such that  $\Phi(B) = U\begin{pmatrix} Y_{11} & Y_{12} \cr Y_{21} & Y_{22} \cr
\end{pmatrix}V$ with $Y_{11} \in \bM_k$.
Because $\Phi(A) \pm \Phi(B)$ are partial isometries,
we see that the Euclidean norm of each of the first $k$ columns of
$\Phi(A)+\Phi(B)$ and $\Phi(A)-\Phi(B)$ is not larger than one.
Thus, $Y_{11}, Y_{21}$ are
zero matrices. Considering the norms of the first $k$ rows of
$\Phi(A)+\Phi(B)$, we see that $Y_{12}$ is the zero matrix as well.
Thus, $\Phi(A), \Phi(B)$ are  disjoint partial isometries in $\bM_{r,s}$.
In general, due to the singular value decomposition, every rectangular matrix can be written
as a linear sum of disjoint rank one partial isometries.  Thus $\Phi$ sends disjoint
partial isometries to disjoint partial isometries.

(b) $\implies (c)$:   $\Phi$ preserves disjointness of
rank one partial isometries, and hence preserves disjointness due to the singular value decomposition. Evidently, it
sends a nonzero partial isometry to a partial isometry.

(c) $\implies$ (e): Because $\Phi$ preserves disjointness,
$\Phi$ has the form described in
Theorem \ref{thm1}. By the fact that $\Phi$
sends a nonzero partial isometry to a partial isometry,
we see that $Q_1, Q_2$ are identity matrices.
So, conditions (a), (b), (c) and (e) are equivalent.

By Lemma \ref{lem:tcp} we have (d) $\implies$ (c).
The implication (e) $\implies$ (d) is also clear.
\end{proof}

Several remarks are in order. Theorem \ref{3.1} and Theorem \ref{thm:tp} are also valid for real linear maps  $\Phi: \bH_n \rightarrow \bM_{r,s}$.
Note that  self-adjoint partial isometries are exactly differences $p-q$ of  two orthogonal projections.
Indeed, we can further assume that the co-domain is
$\bH_r$, i.e., $\Phi: \bH_n \rightarrow \bH_r$.
Then we can arrange $U= V^*$ in  (\ref{s-form-2}) and (\ref{s-form-3}),
at the expenses that $Q_1$, $Q_2$ may have
negative diagonal matrices in (\ref{s-form-2}), and
(\ref{s-form-3}) may look like

\begin{equation*}
\Phi(A) = U \begin{pmatrix}
A \otimes I_{q_1^+} & 0 & 0 & 0 & 0 \cr
0 & - A \otimes I_{q_1^-}  & 0 & 0 & 0 \cr
0& 0 &  A^t\otimes I_{q_2^+} &0&0\cr
0& 0 & 0 & - A^t\otimes I_{q_2^-} &0\cr
0 & 0 &0 & 0 & 0 \cr \end{pmatrix}U^*,
\end{equation*}
where  $q_1^+, q_1^-,q_2^+, q_2^-$ are nonnegative integers and
the zero block matrix in the bottom right corner has size
$(r-((q_1^++q_1^-)m+(q_2^+ +q_2^-)n))\times (r-((q_1^++q_1^-)n+(q_2^+ +q_2^-)m))$.

Theorem \ref{3.1} (a) allows us to obtain the following general
result on linear preserver of functions of
$JB^*$-triple product on matrices.

\begin{corollary} Let $\nu_1, \nu_2$ be scalar
functions on $M_{m,n}$ and $M_{r,s}$
such that
$$
\nu_j(A) = 0\quad\text{if and only if}\quad A = 0
$$
for all $A$ in $M_{m,n}$ or $M_{r,s}$, respectively.
Suppose a linear map $\Phi: M_{m,n} \rightarrow M_{r,s}$ satisfies
\begin{equation}\label{nu1nu2}
\nu_1(\{A,B,C\}) = \nu_2\{\Phi(A),\Phi(B),\Phi(C)\})
\qquad \hbox{ for all } A, B, C \in M_{m,n}.
\end{equation}
Then $\Phi$ has the form $(\ref{s-form-2})$.
\end{corollary}

This corollary can be used to determine the structure of linear preservers
of functions on triple product of matrices easily. We mention a few
examples in the following related to the study in
\cite{CW04, CM05, CLS05,CLP14,HLW08,HLW10,LPS}
and their references.

Suppose a linear map $\Phi: M_{m,n} \rightarrow M_{r,s}$
satisfies (\ref{nu1nu2}), where $\nu_1, \nu_2$ are norms
on matrices. Then $\Phi$ has the form $(\ref{s-form-2})$.
From this, one may easily deduce the conditions on
$U, V, Q_1, Q_2,$ etc. to ensure the converse of the statement.
For example, if $\nu_1, \nu_2$ are the operator norms, then $U, V$ can be
any unitary matrices and the operator norm of $D_1 \oplus D_2$
has to be one.

Suppose $(m,r) = (n,s)$, $\IF = \IC$,
and $\nu_1, \nu_2$ are the numerical radius.
Then  $\Phi: M_n\rightarrow M_r$ satisfies
(\ref{nu1nu2}) if and only if $\Phi$ has the form
(\ref{s-form-2}) with $V = RU^*$ for a diagonal matrix
$R$ such that  $((I_n\otimes Q_1)\oplus (I_n\otimes Q_2)\oplus 0)R$
has numerical radius $1$.
From this, one may further deduce that
when $(m,r) = (n,s)$, $\IF = \IC$,  and $\nu_1, \nu_2$
are the numerical range, $\Phi: M_n \rightarrow M_r$
satisfies (\ref{nu1nu2}) if and only if $\Phi$ has the form
(\ref{s-form-3}) with $V = U^*$.
Similarly, we can treat the linear preservers
$\Phi: M_n\rightarrow M_r$ leaving invariant the pseudo spectral
radius,  pseudo spectrum, and  other types of scalar or non-scalar
functions.

\section{Nonsurjective norm preservers}\label{sec:4}

Denote the singular values of  $A \in \bM_{m,n}$ by
$s_1(A)\ge \cdots \ge s_h(A)$ for $h = \min\{m,n\}$.
For $p > 0$, let
$$S_p(A) = \left(\sum_{j=1}^h s_j(A)^p\right)^{1/p}.$$
If $p \ge 1$, then $S_p(A)$ is known as the \emph{Schatten $p$-norm}.
In particular,  $S_2(A) = (\sum_{j=1}^h s_j(A))^{1/2}
= (\tr(A^*A))^{1/2}$, which is called the \emph{Frobenius norm}, equips $\bM_{m,n}$ as a Hilbert space. For $1\leq p < +\infty$ but $p\ne 2$, a linear operator $\Psi: \bM_{m,n} \rightarrow \bM_{m,n}$
satisfies $S_p(\Psi(A)) = S_p(A)$ for all $A \in \bM_{m,n}$
if and only if
$\Psi$ has the form
$A \mapsto UAV$, or $A \mapsto UA^tV$ in case $m = n$,
for some $U \in \bU_m, V \in \bU_n$
(see, e.g., \cite{CLS05, LT}).

It is more difficult to characterize linear isometries
from $\bM_{m,n}$ to $\bM_{r,s}$ for $(m,n) \ne (r,s)$.
Only very few results are known; see,
for example, \cite{CLP04,LPS}. With Theorem \ref{thm1},
we get the following result.

\begin{theorem} \label{thm3.1}
Suppose $m, n\geq 2$, $p \in (0,2)\cup (2, +\infty)$,
and $\Phi: \bM_{m,n}\rightarrow \bM_{r,s}$ is a linear map.
The following conditions are equivalent.
\begin{itemize}
\item[{\rm (a)}] $S_p(\Phi(A)) = S_p(A)$ for all $A \in \bM_{m,n}$.
\item[
{\rm (b)}] $S_p(\Phi(A)) = S_p(A)$ for all $A \in \bM_{m,n}$ with rank
at most 2.
\item[
{\rm (c)}]
 There are $U \in \bU_r, V\in \bU_s$,
and diagonal matrices $Q_1\in \bM_{q_1}, Q_2\in \bM_{q_2}$
with positive diagonal entries such that $S_p(Q_1 \oplus Q_2) = 1$ and
$$\Phi(A) =  U\begin{pmatrix}
A \otimes Q_1  & 0 & 0 \cr
0 & A^t \otimes Q_2 & 0 \cr
0 & 0 & 0 \cr\end{pmatrix}V \qquad \hbox{ for all } A \in \bM_{m,n}.$$
Here $Q_1$ or $Q_2$ may be vacuous.
\end{itemize}
 \end{theorem}
\begin{proof}
The implications (c) $\implies$ (a) $\implies$ (b) are clear.
For the implication  (b) $\implies$ (c), it follows from a result of McCarthy
\cite[Theorem 2.7]{M} that $\Phi$ preserves disjointness for rank one matrix pairs.
 By Theorem \ref{thm1}, we get the form of $\Phi$. Applying the fact that $S_p(\Phi(E_{11})) = S_p(E_{11})$, we easily deduce that $S_p(Q_1\oplus Q_2)=1$ .
\end{proof}

For $1 \le k \le \min\{m,n\}$,
the \emph{Ky Fan $k$-norm} of $A$ is defined by
$$
F_k(A) = \sum_{j=1}^k s_j(A).
$$
Linear isometries
for the Ky Fan $k$-norm  have been studied.
Seeing Theorem \ref{thm3.1}, one may think that a similar extension
for the Ky Fan $k$-norm can be obtained by similar arguments.
It turns out that this can only be done for the complex case because there are real linear isometries for Ky Fan $k$-norms that do not preserve disjointness; see
\cite{JLL,LT}. This reinforces the fact that proof techniques
for complex matrices may not apply to real matrices, and
it is quite remarkable that a uniform proof
of Theorem \ref{thm1} can be used for both real
and complex matrices. In any event, we have
the following  theorem  supplementing  \cite[Theorem 1.1]{LPS}, in which the
linear map $\Phi: \bM_{m,n}(\IC) \rightarrow \bM_{r,s}(\IC)$ is assumed to satisfy  that
\begin{align*}
F_k(\Phi(A)) = F_{k'}(A), \quad \text{for all}\ A \in \bM_{m,n}(\IC).
\end{align*}

\begin{theorem} \label{thm3.2}
Suppose $2\leq k' \leq \min\{m,n\}$ and $1\leq k \le \min\{r,s\}$.
The following conditions are equivalent  for a linear map
$\Phi: \bM_{m,n}(\IC) \rightarrow \bM_{r,s}(\IC)$.
\begin{enumerate}[(a)]
\item $F_k(\Phi(A)) = F_{k'}(A)$ for all $A \in \bM_{m,n}(\IC)$
with rank at most $2$.

\item There are unitary matrices $U \in \bM_r(\IC)$, $V\in \bM_s(\IC)$ and positive-definite
diagonal matrices $Q_1, Q_2$ (maybe vacuous) of size $q_1, q_2$ such that $k\geq 2(q_1+q_2)$, $Q_1\oplus Q_2$ has trace $1$, and
\begin{align}\label{eq:F-norm}
  \Phi(A)=
U\begin{pmatrix}
 A \otimes Q_{1}  & 0 & 0\cr
0 & A^t \otimes Q_{2}  & 0\cr
0  & 0 & 0\cr\end{pmatrix}V.
\end{align}
\end{enumerate}
\end{theorem}

\noindent
\begin{proof}\rm The implication  (b) $\implies$
(a)  is plain.

(a) $\implies$ (b). By \cite[Lemma 2.2]{LPS}, $\Phi$ preserves disjoint rank one pairs.
By Theorem \ref{thm1},
$\Phi$ carries the form \eqref{eq:F-norm}.
Consider $A_\epsilon=E_{11}+\epsilon E_{22}$ for $0\leq\epsilon< 1$.
Using \eqref{eq:F-norm}, we can assume
$$
\Phi(A_\epsilon)= \lambda_1 A_\epsilon \oplus \lambda_2 A_\epsilon \oplus \cdots\oplus \lambda_q A_\epsilon\oplus 0
$$
for some fixed scalars $\lambda_1\geq \lambda_2 \geq\cdots\geq \lambda_q >0$ with $q=q_1+q_2$.

Suppose $k\leq q$ first.  Since $k'\geq 2$, we  have
\begin{align*}
1+\epsilon &=F_{k'}(A_\epsilon)
= F_k(\lambda_1 A_\epsilon \oplus \lambda_2 A_\epsilon \oplus \cdots\oplus \lambda_q A_\epsilon\oplus 0)\\
&=     \lambda_1 + \lambda_2 + \cdots +  \lambda_{k}, \quad \text{when}\ 0\leq \epsilon\lambda_1 \leq \lambda_{k}.
\end{align*}
This yields a contradiction, because $[0, \lambda_{k}/\lambda_{1}]$ contains infinitely many points $\epsilon$.

Suppose $0< r= k-q<q$.  Then we  have
\begin{align*}
1+\epsilon &=F_{k'}(A_\epsilon)
= F_k(\lambda_1 A_\epsilon \oplus \lambda_2 A_\epsilon \oplus \cdots\oplus \lambda_q A_\epsilon\oplus 0)\\
&= \left\{
  \begin{array}{ll}
    \lambda_1 + \lambda_2 + \cdots +  \lambda_{q}, & \text{when}\  \epsilon=0, \\
    \lambda_1 + \lambda_2 + \cdots +  \lambda_{q} + \epsilon\lambda_1  + \cdots + \epsilon\lambda_{r}, & \text{when}\  \epsilon\lambda_{r+1}\leq \lambda_{q}. \\
  \end{array}
\right.
\end{align*}
This implies $ \lambda_1 + \lambda_2 + \cdots +  \lambda_{q}=1$, and $1+\epsilon =1 + \epsilon \lambda_1  + \cdots + \epsilon\lambda_{r}$ for all $0< \epsilon \leq  \lambda_q/\lambda_{r+1}$.
This  gives us the contradiction that $\lambda_{r+1} = \cdots  = \lambda_{q}=0$.

Hence,   $k\geq 2q$.  In this case, we have
\begin{align*}
1+\epsilon &=F_{k'}(A_\epsilon)
= F_k(\lambda_1 A_\epsilon \oplus \lambda_2 A_\epsilon \oplus \cdots\oplus \lambda_q A_\epsilon\oplus 0)\\
&=     (1+\epsilon)(\lambda_1 + \lambda_2 + \cdots +  \lambda_{q}),\quad \text{when}\  \epsilon \in [0,1).
\end{align*}
This gives $1=\lambda_1 + \lambda_2 + \cdots +  \lambda_{q}$, which equals the trace of $Q_1\oplus Q_2$.
 \end{proof}

\section{Final remarks and future research}\label{sec:5}

It would be interesting to extend our results in Sections 2 and  3
to the (real or complex) linear space $B(H, K)$ of bounded linear operators between
infinite dimensional Banach spaces $H$ and $K$, or to general
JB*-triples.
Our approach depends on the singular value decomposition of matrices,
which is a finite dimensional feature.
New techniques will be needed to extend our results.

\medskip
To conclude the paper, we list
several comments and questions concerning the results in Section 4.

\begin{enumerate}
\item As pointed out in \cite{CLP04}, the problem for the operator norm, i.e.,
Ky Fan $1$-norm, is difficult.

\item
Many  real linear isometries for Ky Fan $k$-norms also preserve
disjointness (although there are exceptions). It would be nice to
investigate a version of
Theorem \ref{thm3.2}  such that
the conclusion also hold for real matrices.

\item For any linear isometry which preserves disjoint rank one pairs,  we can apply Theorem \ref{thm1}.
It is interesting to characterize such norms other than the Schatten $p$-norms and the Ky Fan $k$-norms.
Suggested by the asserted form \eqref{eq:F-norm}, we should put emphasis on unitarily invariant norms.

\item We have similar results for real symmetric and complex Hermitian matrices.
Besides $S_p(A)$ and $F_k(A)$, can we do it for the
$k$-numerical radius on
Hermitian matrices $\bH_n$ defined by
$$w_k(A) =\max\{ \tr(AR): R^* = R = R^2, \tr R = k\}?$$

\item In fact, one can also ask for characterizations of
$k$-numerical radius preservers $\Phi: \bM_n \rightarrow \bM_r$.

\item One may consider linear preservers or non-linear preservers for other types of norms or functions
on rectangular matrices, Hermitian, symmetric, or skew-symmetric matrix spaces that are related to disjointness preserving maps.

\end{enumerate}

\section*{Acknowledgment}

Li is an affiliate member of the Institute for Quantum Computing,
University of Waterloo. He is an honorary professor of Shanghai
University. His research was supported by USA NSF grant DMS 1331021,
Simons Foundation Grant 351047, and NNSF of China Grant 11571220. This
research
was started when he visited Taiwan in 2018 supported by grants from Taiwan MOST.
He would like to express
his gratitude to the hospitality of several institutions, including the
Academia Sinica,
National Taipei University of Science and Technology,
National Chung Hsing University, and National Sun
Yat-sen University. He would also like to thank Dr.\ Daniel Puzzuoli
for some helpful discussions.

Tsai, Wang and Wong are supported by Taiwan MOST grants
105-2115-M-027-002-MY2,
106-2115-M-005-001-MY2 and  106-2115-M-110-006-MY2,
respectively.

\end{document}